\documentclass[a4paper,11pt,oneside]{amsart}

\usepackage[applemac]{inputenc}
\usepackage[british]{babel}
\usepackage{fancyhdr}
\usepackage{geometry}
\usepackage{setspace}
\usepackage{bbm}
\usepackage{amssymb}
\usepackage[all]{xy}

\usepackage{color}
\definecolor{citation}{rgb}{0.2,0.3,1}
\usepackage[final,colorlinks,citecolor=citation,pagebackref,linktocpage]{hyperref}

\newtheoremstyle{par}{1ex}{2ex}{\rm}{}{\bfseries}{}{0.8em}{\thmnumber{(#2)}}
\newtheoremstyle{thm}{1ex}{2ex}{\itshape}{}{\bfseries}{}{0.9em}{\thmnumber{(#2)}\thmname{ #1}\thmnote{ (#3)}}
\newtheoremstyle{ex}{1ex}{2ex}{\rm}{}{\bfseries}{}{0.8em}{\thmnumber{(#2)}\thmname{ #1}}

\theoremstyle{par}
\newtheorem{no}{}[section]

\theoremstyle{thm}
\newtheorem{prop}[no]{Proposition}
\newtheorem{cor}[no]{Corollary}
\newtheorem{thm}[no]{Theorem}

\newcommand{\N}{\mathbbm{N}}
\newcommand{\Z}{\mathbbm{Z}}
\newcommand{\Q}{\mathbbm{Q}}
\newcommand{\C}{{\sf C}}
\newcommand{\D}{{\sf D}}
\newcommand{\E}{{\sf E}}
\newcommand{\NN}{{\bf N}}
\newcommand{\LL}{{\bf L}}
\newcommand{\dfgl}{\mathrel{\mathop:}=}
\newcommand{\ab}{{\sf Ab}}
\newcommand{\grmod}{{\sf Mod}}
\newcommand{\hm}[4]{{\rm Hom}^{#1}_{#2}(#3,#4)}
\newcommand{\hme}{{\rm hom}}
\newcommand{\Id}{{\rm Id}}
\newcommand{\h}{\widetilde{h}}
\newcommand{\hsub}{\underline{h}}
\newcommand{\eps}{\varepsilon}
\newcommand{\thet}{\vartheta}
\newcommand{\ps}{_{[\psi]}}
\newcommand{\sq}{\hskip2pt\raisebox{.225ex}{\rule{.8ex}{.8ex}\hskip2pt}}
\newcommand{\ia}{\mathfrak{a}}

\DeclareMathOperator{\ke}{Ker}

\newdir{ >}{{}*!/-10pt/\dir{>}}

\begin{document}

\newcommand{\leadingzero}[1]{\ifnum #1<10 0\the#1\else\the#1\fi} 
\renewcommand{\today}{\the\year\leadingzero{\month}\leadingzero{\day}}

\title{Graded change of ring}
\author{Fred Rohrer}
\address{Grosse Grof 9, 9470 Buchs, Switzerland}
\email{fredrohrer@math.ch}

\begin{abstract}
We investigate scalar restriction, scalar extension, and scalar coextension functors for graded modules, including their interplay with coarsening functors, graded tensor products, and graded Hom functors. This leads to several characterisations of epimorphisms of graded rings.
\end{abstract}

\maketitle\thispagestyle{fancy}


\section*{Introduction}

By a \textit{group} or a \textit{ring} we always mean a commutative group or a commutative ring. A morphism of rings $h\colon R\rightarrow S$ induces certain change of ring functors, namely the ubiquitous scalar extension \[h^*\colon\grmod(R)\rightarrow\grmod(S)\] and scalar restriction \[h_*\colon\grmod(S)\rightarrow\grmod(R)\] (\cite[II.1.13, II.5]{a}), as well as the slightly less prominent scalar coextension \[\h\colon\grmod(R)\rightarrow\grmod(S)\] (\cite[II.5.1 Remarque 4]{a}, \cite[II.6]{ce}, \cite[5.1]{lambek}). If $h\colon R\rightarrow S$ is a morphism of $G$-graded rings for some group $G$, then analogue functors between the categories of $G$-graded modules can be defined. It is the goal of this article to comprehensively study the three functors $h^*$, $h_*$ and $\h$ in such a graded setting. In accordance with the yoga of coarsening (cf.\ \cite{cihf}), we consider throughout an epimorphism of groups $\psi\colon G\rightarrow H$ and investigate the behaviour of the above three functors with respect to the coarsening functor \[\bullet\ps\colon\grmod^G(R)\rightarrow\grmod^H(R\ps).\]

Most of the results in this article are rather easy to prove and moreover not astonishing at all. Exceptions might be the various characterisations of epimorphisms of (graded) rings (\ref{d70}) and its corollary about preservation of epimorphisms of graded rings by coarsening functors (\ref{d80}).) However, when working with graduations it seems desirable to have a such a comprehensive treatment in written form, which to the authors knowledge was not available previously.\medskip

In Section 1 we collect some preliminaries on graded modules and coarsening functors. Most of them may be well-known, but for lack of reference and ease of readability we often provide detailed explanations and full proofs.

In Section 2 we define the change of ring functors $h_*$, $h^*$ and $\h$ and investigate their behaviour under coarsening (\ref{b10}, \ref{b20}, \ref{b30}). One should note that since Hom functors need not commute with coarsening in general (\cite[Section 3]{cihf}), neither need $\h$. As a byproduct we get the graded version of the Hom-tensor adjunction (\ref{b50}), and as an application thereof we derive some properties of the canonical morphisms \[\hm{G}{R}{L}{M}\otimes_RN\rightarrow\hm{G}{R}{L}{M\otimes_RN}\] and \[M\otimes_RN\rightarrow\hm{G}{R}{\hm{G}{R}{M}{R}}{N},\] where Hom is understood to be taken in the category of $G$-graded $R$-modules (\ref{b70}, \ref{b80}).

Section 3 starts with recognising an adjunction $(h^*,h_*)$ whose counit is an epimorphism, and whose unit is a mono-, epi- or isomorphism if and only if $h$, considered as a morphism of $G$-graded $R$-modules, is pure, an epimorphism, or an isomorphism, and an adjunction $(h_*,\h)$ whose unit is a monomorphism, and whose counit is a mono-, epi- or isomorphism if and only if $h$, considered as a morphism of $G$-graded $R$-modules, is an epimorphism, a section, or an isomorphism (\ref{c20}, \ref{c40}, \ref{c50}). Then, we have a look at exactness properties of change of ring functors. The aformentioned adjunctions cause $h_*$ and $h^*$ to commute with inductive limits and $h_*$ and $\h$ to commute with projective limits. Additionally, we show that the following statements are equivalent: (i) $\h$ commutes with inductive limits; (ii) $h^*$ commutes with projective limits; (iii) $h_*(S)$ is projective and of finite type; (iv) the adjoint triple $(h^*,h_*,\h)$ can be extended to the left and the right (\ref{c79}, \ref{c120}). Moreover, in this case we describe these further adjoints (\ref{c140}) and get as an application a graded version of Morita's characterisation of the coincidence of scalar extension and coextension (\cite{morita}): We have $h^*\cong\h$ if and only if $h_*(S)$ is projective and of finite type and $\h(R)\cong S$ (\ref{c150}).

Section 4 is about the interplay of change of ring functors with tensor products and Hom functors. We construct and study an isomorphism \[\delta^h\colon h^*(\bullet)\otimes_Sh^*(\sq)\rightarrow h^*(\bullet\otimes_R\sq),\] an epimorphism \[\gamma^h\colon h_*(\bullet)\otimes_Rh_*(\sq)\rightarrow h_*(\bullet\otimes_S\sq),\] and a morphism \[\eps^h\colon\h(\bullet)\otimes_S\h(\sq)\rightarrow\h(\bullet\otimes_R\sq)\] that need neither be a mono- nor an epimorphism (\ref{d10}, \ref{d20}, \ref{d25}). We also construct a monomorphism \[\eta^h\colon h_*(\hm{G}{S}{\bullet}{\sq})\rightarrow\hm{G}{R}{h_*(\bullet)}{h_*(\sq)}\] and a morphism \[\thet^h\colon h^*(\hm{G}{R}{\bullet}{\sq})\rightarrow\hm{G}{S}{h^*(\bullet)}{h^*(\sq)}\] that need neither be a mono- nor an epimorphism (\ref{d30}, \ref{d40}). (Unfortunately, the\linebreak author was not able to find a reasonable morphism between $\h(\hm{G}{R}{\bullet}{\sq})$ and\linebreak $\hm{G}{S}{\h(\bullet)}{\h(\sq)}$.) Finally, we show that the following statements are equivalent:\linebreak (i) $h$ is an epimorphism of $G$-graded rings; (ii) $\gamma^h$ is an isomorphism; (iii) $\eta^h$ is an isomorphism; (iv) the counit of $(h^*,h_*)$ is an isomorphism; (v) the unit of $(h_*,\h)$ is an isomorphism (\ref{d70}). This contains a graded version of Roby's characterisation of epimorphisms (\cite{roby}). As a corollary we get that coarsening functors preserve epimorphisms of graded rings (\ref{d80}).\medskip

\textbf{Notation.} In general, notation and terminology follow Bourbaki's \textit{\'El\'ements de math\'ematique.} Additionally, we denote by $\ab$ the category of groups and by $\grmod^G(R)$ the category of $G$-graded $R$-modules (for a group $G$ and a $G$-graded ring $R$). Further notation and terminology concerning graded rings and modules follow \cite{cihf}. In particular, for an epimorphism $\psi\colon G\twoheadrightarrow H$ of groups we denote by $\bullet\ps$ the $\psi$-coarsening functor from the category of $G$-graded rings to the category of $H$-graded rings as well as the $\psi$-coarsening functor from $\grmod^G(R)$ to $\grmod^H(R\ps)$ (for a $G$-graded ring $R$).

Throughout the following, we fix an epimorphism of groups $\psi\colon G\twoheadrightarrow H$ and a morphism of $G$-graded rings $h\colon R\rightarrow S$. If we consider $h$ as a morphism of $G$-graded $R$-modules (from $R$ to $h_*(S)$, cf. \ref{b10}), then we denote it by $\hsub$.


\section{Preliminaries on graded modules}

\begin{no}\label{a05}
Even though the notion of adjoint functors is crucial for our investigation, we use only the very modest amount of results from category theory recalled below.\smallskip

A) Let $F\colon\C\rightarrow\D$ and $G\colon\D\rightarrow\C$ be functors, and let $\rho\colon\Id_{\C}\rightarrow G\circ F$ and $\sigma\colon F\circ G\rightarrow\Id_{\D}$ be morphisms of functors. If $(\sigma\circ F)\circ(F\circ\rho)=\Id_F$ and $(G\circ\sigma)\circ(\rho\circ G)=\Id_G$, then there is an adjunction $(F,G)$, and $\rho$ and $\sigma$ are called its unit and its counit (\cite[1.5.3]{ks}).\smallskip

B) Left (or right) adjoint functors are unique up to unique isomorphisms (\cite[1.5.3]{ks}).\smallskip

C) Let $F\colon\C\rightarrow\D$, $F'\colon\D\rightarrow\E$, $G\colon\D\rightarrow\C$ and $G'\colon\E\rightarrow\D$ be functors. If there are adjunctions $(F,G)$ and $(F',G')$, then there is an adjunction $(F'\circ F,G\circ G')$ (\cite[1.5.5]{ks}).\smallskip

D) Let $F\colon\C\rightarrow\D$ be a functor. If $F$ has a left (or right) adjoint, then it commutes with projective (or inductive) limits (\cite[2.1.10]{ks}). The converse holds if $\C$ is an AB5 category with a generator (\cite[9.6.4]{ks}).\smallskip

E) Let $(F,G)$ be an adjunction with unit $\rho$ and counit $\sigma$. If $F$ is faithful then $\rho$ is a monomorphism; if $G$ is faithful then $\sigma$ is an epimorphism (\cite[16.5.3]{schubert2}).\smallskip

F) A functor commutes with projective (or inductive) limits if and only if it is left (or right) exact and commutes with product (or coproducts) (\cite[I.6.4.4]{ilf}).
\end{no}

\begin{no}
The category $\grmod^G(R)$ is abelian, fulfils AB5 and AB4$^*$, and has a projective generator and an injective cogenerator (\cite[A.I.1]{nvo1}). If $g\in G$, then the $g$-shift functor $\bullet(g)\colon\grmod^G(R)\rightarrow\grmod^G(R)$ is an isomorphism of categories (with inverse $\bullet(-g)$) and thus commutes with inductive and projective limits. These basic facts will be used freely throughout the following.
\end{no}

\begin{no}\label{a10}
The $\psi$-coarsening functor $\bullet\ps\colon\grmod^G(R)\rightarrow\grmod^H(R\ps)$ is faithful, conservative, exact, and commutes with inductive limits; it commutes with projective limits if and only if $\ke(\psi)$ is finite (\cite[Proposition 1.2, Theorem 1.3]{cihf}, \ref{a05} D)). In general, for a projective system $F\colon I\rightarrow\grmod^G(R)$ there is a canonical morphism \[(\varprojlim\nolimits_{i\in I}F(i))\ps\rightarrow\varprojlim\nolimits_{i\in I}(F(i)\ps)\] in $\grmod^H(R\ps)$.
\end{no}

\begin{no}\label{a20}
It follows on use of \cite[A.I.2.1]{nvo1} that a morphism $u$ in $\grmod^G(R)$ is a section (or a retraction) if and only if $u\ps$ is so. Therefore, a short exact sequence $\mathbbm{S}$ in $\grmod^G(R)$ is split if and only if the short exact sequence $\mathbbm{S}\ps$ in $\grmod^H(R\ps)$ is split.
\end{no}

\begin{no}\label{a30}
A $G$-graded $R$-module $M$ is called \textit{free (of finite rank)} if $M\cong\bigoplus_{i\in I}R(g_i)$ for some (finite) family $(g_i)_{i\in I}$ in $G$. If $M$ is free (of finite rank) then so is $M\ps$; the converse need not hold (\cite[A.I.2.6.2]{nvo1}).

A $G$-graded $R$-module $M$ is called \textit{of finite type} (or \textit{of finite presentation}) if there is an exact sequence $L\rightarrow M\rightarrow 0$ (or $L'\rightarrow L\rightarrow M\rightarrow 0$) in $\grmod^G(R)$ with $L$ (and $L'$) free of finite rank. This holds if and only if $M\ps$ is of finite type (or of finite presentation). Indeed, $\bullet\ps$ preserves both properties by the first paragraph and \ref{a10}. For the converses we may suppose that $H=0$. If $M_{[0]}$ is of finite type then the set of homogeneous components of the elements of a finite generating set of $M_{[0]}$ is a finite homogeneous generating set of $M$, hence $M$ is of finite type. If $M_{[0]}$ is of finite presentation then by the above there is an epimorphism $p\colon L\twoheadrightarrow M$ in $\grmod^G(R)$ with $L$ free of finite rank, hence $\ke(p)_{[0]}=\ke(p_{[0]})$ is of finite type by \cite[X.1.4 Proposition 6]{a} and \ref{a10}, thus so is $\ke(p)$ by what we have already shown. Therefore, $M$ is of finite presentation.
\end{no}

\begin{no}\label{a40}
A) The $G$-graded tensor product bifunctor \[\bullet\otimes_R\sq\colon\grmod^G(R)^2\rightarrow\grmod^G(R)\] maps a pair $(M,N)$ of $G$-graded $R$-modules to the $G$-graded $R$-module \[M\otimes_RN=\bigoplus_{g\in G}\bigl(\bigoplus_{g'+g''=g}M_{g'}\otimes_{R_0}M_{g''}\bigr)\] and commutes with $\psi$-coarsening, i.e., there is an isomorphism of functors \[(\bullet\otimes_R\sq)\ps\cong(\bullet\ps)\otimes_{R\ps}(\sq\ps).\] For $g\in G$ there are isomorphisms \[(\bullet(g))\otimes_R\sq\cong\bullet\otimes_R(\sq(g))\cong(\bullet\otimes_R\sq)(g).\]

B) A $G$-graded $R$-module $M$ is called \textit{flat} if $M\otimes_R\bullet\colon\grmod^G(R)\rightarrow\grmod^G(R)$ is exact. By \cite[A.I.2.18]{nvo1}, $M$ is flat if and only if $M\ps$ is so.\smallskip

C) A morphism $u\colon M\rightarrow N$ in $\grmod^G(R)$ is called \textit{pure} if $u\otimes_R\bullet\colon M\otimes_R\bullet\rightarrow N\otimes_R\bullet$ is a monomorphism. By \cite[A.I.2.20]{nvo1}, $u$ is pure if and only if $u\ps$ is so.
\end{no}

\begin{no}\label{a50}
A) The $G$-graded Hom bifunctor \[\hm{G}{R}{\bullet}{\sq}\colon\grmod^G(R)^2\rightarrow\grmod^G(R)\] maps a pair $(M,N)$ of $G$-graded $R$-modules to the $G$-graded $R$-module \[\hm{G}{R}{M}{N}=\bigoplus_{g\in G}\hm{}{\grmod^G(R)}{M}{N(g)},\] and there is a canonical monomorphism \[\beta^\psi\colon\hm{G}{R}{\bullet}{\sq}\ps\rightarrowtail\hm{H}{R\ps}{\bullet\ps}{\sq\ps}.\] If $M$ is a $G$-graded $R$-module, then $\beta^\psi_{M,N}$ is an isomorphism for every $G$-graded $R$-module $N$ if and only if $\ke(\psi)$ is finite or $M$ is small (\cite[3.6]{cihf}, cf. \ref{a60} A)). For $g\in G$ there are isomorphisms \[\hm{G}{R}{\bullet(g)}{\sq}\cong\hm{G}{R}{\bullet}{\sq(-g)}\cong\hm{G}{R}{\bullet}{\sq}(-g).\]

B) A $G$-graded $R$-module $M$ is called \textit{projective} if $\hm{}{\grmod^G(R)}{M}{\bullet}\colon\grmod^G(R)\rightarrow\ab$ is exact, and this holds if and only if $\hm{G}{R}{M}{\bullet}\colon\grmod^G(R)\rightarrow\grmod^G(R)$ is exact. Furthermore, $M$ is projective (and of finite type) if and only if it is a direct summand of a free $G$-graded $R$-module (of finite rank). In particular, projective $G$-graded $R$-modules are flat. By \cite[I.2.2]{nvo1}, $M$ is projective if and only if $M\ps$ is so. It follows thus from \ref{a30}, \ref{a40} B), and the corresponding ungraded statement (\cite[X.1.5]{a}), that a $G$-graded $R$-module is flat and of finite presentation if and only if it is projective and of finite type.
\end{no}

\begin{no}\label{a60}
A) Let $M$ be a $G$-graded $R$-module. For a family $\NN=(N_j)_{j\in J}$ of $G$-graded $R$-modules there are canonical monomorphisms \[\lambda_{M,\NN}\colon\bigoplus_{j\in J}\hm{}{\grmod^G(R)}{M}{N_j}\rightarrowtail\hm{}{\grmod^G(R)}{M}{\bigoplus_{j\in J}N_j}\] in $\ab$ and \[\Lambda_{M,\NN}\colon\bigoplus_{j\in J}\hm{G}{R}{M}{N_j}\rightarrowtail\hm{G}{R}{M}{\bigoplus_{j\in J}N_j}\] in $\grmod^G(R)$, both with $(u_j)_{j\in J}\mapsto(x\mapsto(u_j(x))_{j\in J})$ for $u_j\in\hm{}{\grmod^G(R)}{M}{N_j}$ or $u_j\in\hm{G}{R}{M}{N_j}^\hme$, resp., for $j\in J$ and $x\in M^\hme$. The $G$-graded $R$-module $M$ is called \textit{small} if $\lambda_{M,\NN}$ is an isomorphism for every $\NN$, and this holds if and only if $\Lambda_{M,\NN}$ is an isomorphism for every $\NN$. If $M$ is of finite type then it is small, but the converse need not hold (\cite[2$^\circ$, 5$^\circ$]{rentschler}); it does hold if $M$ is projective by \ref{a30}, \ref{a50} B), and the corresponding ungraded statement in the proof of \cite[II.1.2]{bass}. Setting $\NN\ps\dfgl((N_j)\ps)_{j\in J}$ there is a commutative diagram \[\xymatrix@C60pt{(\bigoplus_{j\in J}\hm{G}{R}{M}{N_j})\ps\ar@{ >->}[r]^{(\Lambda_{M,\NN})\ps}\ar@{ >->}[d]&\hm{G}{R}{M}{\bigoplus_{j\in J}N_j}\ps\ar@{ >->}[d]^{\beta^\psi_{M,\bigoplus_{j\in J}N_j}}\\\bigoplus_{j\in J}\hm{H}{R\ps}{M\ps}{(N_j)\ps}\ar@{ >->}[r]^{\Lambda_{M\ps,\NN\ps}}&\hm{H}{R\ps}{M\ps}{\bigoplus_{j\in J}(N_j)\ps}}\] in $\grmod^H(R\ps)$, where the left vertical morphism is induced by $(\beta^\psi_{M,N_j})_{j\in J}$. If $M$ is small or $\ke(\psi)$ is finite, then both vertical morphisms are isomorphisms (\ref{a10}, \ref{a50} A)). Furthermore, $M$ is small if and only if $M\ps$ is so (\cite[3.2]{cihf}).\smallskip

B) The $G$-graded ring $R$ is called \textit{steady} if a $G$-graded $R$-module is small if and only if it is of finite type. If $R$ is noetherian then it is steady, but the converse need not hold (\cite[3.5]{gpmn}, \cite[7$^\circ$, 10$^\circ$]{rentschler}). If $R$ is steady and $h$ is surjective, then $S$ is steady (\cite[1.9]{ct}). If $R\ps$ is steady then so is $R$; the converse holds if $\ke(\psi)$ is finite (\cite[3.3]{cihf}).\smallskip

C) Let $M$ and $N$ be $G$-graded $R$-modules. If $M$ and $N$ are small then so is $M\otimes_RN$ by A), \ref{a40} A) and the corresponding ungraded statement (\cite[1.4]{gpmn}). Conversely, if $M\otimes_RN$ is small then neither $M$ nor $N$ need be small; an (ungraded) counterexample is given by the steady ring $\Z$ and the non-small $\Z$-modules $\Q$ and $(\Z/2\Z)^{\oplus\N}$, for $\Q\otimes_{\Z}(\Z/2\Z)^{\oplus\N}=0$ is small.\smallskip

D) Let $M$ and $N$ be $G$-graded $R$-modules. If $\hm{G}{R}{M}{N}$ is small then neither $M$ nor $N$ need be so. An (ungraded) counterexample is given by the steady ring $\Z$, a finite group $A$, and the non-small $\Z$-modules $A^{\oplus\N}$ and $\Q$, for $\hm{}{\Z}{A^{\oplus\N}}{\Q}\cong\hm{}{\Z}{A}{\Q}^{\N}=0$ is small. Conversely, if $M$ and $N$ are small then $\hm{G}{R}{M}{N}$ need not be small.\footnote{For a positive result see \ref{a140}.} For an (ungraded) counterexample, we consider a field $K$ and the local ring $R=K[(X_i)_{i\in\N}]/\langle X_iX_j\mid i,j\in\N\rangle$ whose maximal ideal $\mathfrak{m}$ is not of finite type. Then, $R$ is steady (\cite[10$^\circ$]{rentschler}, B)) and the $R$-module $M=R\oplus R/\mathfrak{m}$ is small, but $\hm{}{R}{M}{M}\cong R\oplus R/\mathfrak{m}\oplus\mathfrak{m}\oplus R/\mathfrak{m}$ has a direct summand that is not small.
\end{no}

\begin{no}\label{a70}
A) Let $M$ be a $G$-graded $R$-module. For a family $\NN=(N_j)_{j\in J}$ of $G$-graded $R$-modules there is a canonical morphism $\kappa_{M,\NN}\colon(\prod_{j\in J}N_j)\otimes_RM\rightarrow\prod_{j\in J}(N_j\otimes_RM)$ in $\grmod^G(R)$ with $(x_j)_{j\in J}\otimes y\mapsto(x_j\otimes y)_{j\in J}$ for $(x_j)_{j\in J}\in(\prod_{j\in J}N_j)^\hme$ and $y\in M^\hme$ that need be neither a mono- nor an epimorphism (\cite[059I]{stacks}). Setting $\NN\ps\dfgl((N_j)\ps)_{j\in J}$ there is a commutative diagram \[\xymatrix@C50pt@R15pt{((\prod_{j\in J}N_j)\otimes_RM)\ps\ar[r]^{(\kappa_{M,\NN})\ps}\ar[d]&(\prod_{j\in J}(N_j\otimes_RM))\ps\ar[d]\\(\prod_{j\in J}(N_j)\ps)\otimes_{R\ps}M\ps\ar[r]^{\kappa_{\NN\ps,M\ps}}&\prod_{j\in J}((N_j)\ps\otimes_{R\ps}M\ps)}\] in $\grmod^H(R\ps)$ where the vertical morphisms are induced by the canonical ones. If $\ke(\psi)$ is finite then both vertical morphisms are isomorphisms (\ref{a10}, \ref{a40} A)).\smallskip

B) If $M$ is free of finite rank then $\kappa_{M,\NN}$ is an isomorphism for every family $\NN$ of $G$-graded $R$-modules. Indeed, for a finite family $(g_i)_{i\in I}$ in $G$ and a family $\NN=(N_j)_{j\in J}$ of $G$-graded $R$-modules we have a commutative diagram \[\xymatrix@C70pt@R15pt{(\prod_{j\in J}N_j)\otimes_R(\bigoplus_{i\in I}R(g_i))\ar[r]^{\kappa_{L,\NN}}\ar[dd]_{\cong}&\prod_{j\in J}(N_j\otimes_R(\bigoplus_{i\in I}R(g_i)))\ar[d]^{\cong}\\&\prod_{j\in J}\bigoplus_{i\in I}(N_j\otimes_R(R(g_i)))\ar[d]_a^{\cong}\\\bigoplus_{i\in I}((\prod_{j\in J}N_j)\otimes_R(R(g_i)))\ar[r]^{\bigoplus_{i\in I}\kappa_{R(g_i),\NN}}\ar[d]_{\cong}&\bigoplus_{i\in I}\prod_{j\in J}(N_j\otimes_R(R(g_i)))\ar[d]^{\cong}\\\bigoplus_{i\in I}(\prod_{j\in J}N_j)(g_i)\ar[r]^\cong&\bigoplus_{i\in I}\prod_{j\in J}(N_j(g_i))}\] in $\grmod^G(R)$ where the unmarked morphisms are the canonical ones and $a$ is an isomorphism because $J$ is finite (\ref{a40} A)). This yields the claim.
\end{no}

\begin{prop}\label{a90}
Let $M$ be a $G$-graded $R$-module.

a) The following statements are equivalent: (i) $\kappa_{M,\NN}$ is an epimorphism for every family $\NN$ of $G$-graded $R$-modules; (ii) $\kappa_{M,\NN}$ is an epimorphism for every family $\NN$ of flat $G$-graded $R$-modules; (iii) $M$ is of finite type.

b) The following statements are equivalent: (i) $\kappa_{M,\NN}$ is an isomorphism for every family $\NN$ of $G$-graded $R$-modules; (ii) $\kappa_{M,\NN}$ is an isomorphism for every family $\NN$ of flat $G$-graded $R$-modules; (iii) $M$ is of finite presentation.
\end{prop}

\begin{proof}
a) Suppose that (ii) holds, so that the map \[R^{M^\hme}\otimes_RM\rightarrow M^{M^\hme},\;\sum_{j=1}^n((b^{(j)}_m)_{m\in M^\hme}\otimes a^{(j)})\mapsto(b_m^{(j)}a^{(j)})_{m\in M^\hme}\] is surjective. There exist $n\in\N$, $b^{(j)}_m\in R$ for $j\in[1,n]$ and $m\in M^\hme$, and $a^{(j)}\in M$ for $j\in[1,n]$ with $(m)_{m\in M^\hme}=(\sum_{j=1}^nb^{(j)}_ma^{(j)})_{m\in M^\hme}$. Hence, if $m\in M^\hme$ then $m=\sum_{j=1}^nb^{(j)}_ma^{(j)}$. This shows that $M=\langle a^{(j)}\mid j\in[1,n]\rangle_R$, thus (iii) holds.

Suppose that (iii) holds, so that we have an epimorphism $L\twoheadrightarrow M$ in $\grmod(R)$ with $L$ free of finite rank. Let $\NN=(N_j)_{j\in J}$ be a family of $G$-graded $R$-modules. By \ref{a70} B) we have a commutative diagram \[\xymatrix{(\prod_{j\in J}N_j)\otimes_RL\ar@{->>}[r]\ar[d]^\cong_{\kappa_{L,\NN}}&(\prod_{j\in J}N_j)\otimes_RM\ar[d]^{\kappa_{M,\NN}}\\\prod_{j\in J}(N_j\otimes_RL)\ar@{->>}[r]&\prod_{j\in J}(N_j\otimes_RM)}\] in $\grmod^G(R)$, implying that $\kappa_{M,\NN}$ is an epimorphism and thus (i).

b) Suppose that (ii) holds. By a) there is an exact sequence \[0\longrightarrow K\longrightarrow L\longrightarrow M\longrightarrow 0\] in $\grmod^G(R)$ with $L$ free of finite rank. Let $\NN=(N_j)_{j\in J}$ be a family of flat $G$-graded $R$-modules. By (ii) and a) we have a commutative diagram with exact rows \[\xymatrix{&(\prod_{j\in J}N_j)\otimes_RK\ar[r]\ar[d]_{\kappa_{K,\NN}}&(\prod_{j\in J}N_j)\otimes_RL\ar[r]\ar@{->>}[d]_{\kappa_{L,\NN}}&(\prod_{j\in J}N_j)\otimes_RM\ar[r]\ar[d]^\cong_{\kappa_{M,\NN}}&0\\0\ar[r]&\prod_{j\in J}(N_j\otimes_R K)\ar[r]&\prod_{j\in J}(N_j\otimes_R L)\ar[r]&\prod_{j\in J}(N_j\otimes_R M)\ar[r]&0}\] in $\grmod^G(R)$. The Snake Lemma (\cite[010H]{stacks}) implies that $\kappa_{K,\NN}$ is an epimorphism, hence $K$ is of finite type by a), and thus (iii) holds.

Suppose that (iii) holds, so that we have an exact sequence $L'\longrightarrow L\longrightarrow M\rightarrow 0$ in $\grmod^G(R)$ with $L$ and $L'$ free of finite rank. Let $\NN=(N_j)_{j\in J}$ be a family of $G$-graded $R$-modules. By \ref{a70} B) we have a commutative diagram with exact rows \[\xymatrix{(\prod_{j\in J}N_j)\otimes_RL'\ar[r]\ar[d]^\cong_{\kappa_{L',\NN}}&(\prod_{j\in J}N_j)\otimes_RL\ar[r]\ar[d]^\cong_{\kappa_{L,\NN}}&(\prod_{j\in J}N_j)\otimes_RM\ar[r]\ar[d]_{\kappa_{M,\NN}}&0\\\prod_{j\in J}(N_j\otimes_RL')\ar[r]&\prod_{j\in J}(N_j\otimes_RL)\ar[r]&\prod_{j\in J}(N_j\otimes_RM)\ar[r]&0}\] in $\grmod^G(R)$. The Five Lemma (\cite[05QB]{stacks}) implies that $\kappa_{M,\NN}$ is an isomorphism and thus (i).
\end{proof}

\begin{cor}
Let $M$ be a $G$-graded $R$-module. Then, $\kappa_{M,\NN}$ is an epimorphism (or isomorphism) for every family $\NN$ of $G$-graded $R$-modules if and only if $\kappa_{M\ps,\NN}$ is so for every family $\NN$ of $H$-graded $R\ps$-modules.
\end{cor}

\begin{proof}
Immediately from \ref{a30} and \ref{a90}.
\end{proof}

\begin{no}
For $G$-graded $R$-modules $L$, $M$ and $N$ there is a morphism \[\tau_{L,M,N}\colon L\otimes_R\hm{G}{R}{M}{N}\rightarrow\hm{G}{R}{\hm{G}{R}{L}{M}}{N}\] in $\grmod^G(R)$ with $x\otimes u\mapsto(v\mapsto u(v(x)))$ for $x\in L^\hme$, $u\in\hm{G}{R}{M}{N}^\hme$ and $v\in\hm{G}{R}{L}{M}^\hme$ that is natural in $L$, $M$ and $N$. In particular, for a $G$-graded $R$-module $L$ the morphism $\tau_{L,R,R}$ composed with the canonical isomorphism $L\cong L\otimes_R\hm{G}{R}{R}{R}$ yields a morphism \[\tau_L\colon L\rightarrow\hm{G}{R}{\hm{G}{R}{L}{R}}{R},\;x\mapsto(u\mapsto u(x))\] in $\grmod^G(R)$ that is natural in $L$.
\end{no}

\begin{prop}\label{a120}
Let $L$, $M$ and $N$ be $G$-graded $R$-modules. If $L$ is projective then $\tau_{L,M,N}$ is a monomorphism. If $L$ is projective and of finite type then $\tau_{L,M,N}$ is an isomorphism.\footnote{For $G=0$ this generalises \cite[II.4 Exercice 6 b)]{a}.}
\end{prop}

\begin{proof}
Let $L'$ be a further $G$-graded $R$-module. Setting $P\dfgl\hm{G}{R}{M}{N}$, $Q\dfgl\hm{G}{R}{L}{M}$ and $Q'\dfgl\hm{G}{R}{L'}{M}$, there is a commutative diagram \[\xymatrix@C80pt@R15pt{
(L\oplus L')\otimes_RP\ar[r]^(.42){\tau_{L\oplus L',M,N}}\ar[d]_\cong&\hm{G}{R}{\hm{G}{R}{L\oplus L'}{M}}{N}\ar[d]^\cong\\
(L\otimes_RP)\oplus(L'\otimes_RP)\ar[r]^(.46){\tau_{L,M,N}\oplus\tau_{L',M,N}}&\hm{G}{R}{Q}{N}\oplus\hm{G}{R}{Q'}{N}}\] in $\grmod^G(R)$ where the vertical morphisms are the canonical ones. Thus, $\tau_{L\oplus L',M,N}$ is an isomorphism if and only if $\tau_{L,M,N}$ and $\tau_{L',M,N}$ are so. So, for the first (or second) claim we can henceforth suppose that $L$ is free (of finite rank) (\ref{a50} B)). If $L=\bigoplus_{i\in I}R(g_i)$ for a family $(g_i)_{i\in I}$ in $G$, then there is a commutative diagram \[\xymatrix@C50pt@R15pt{(\bigoplus_{i\in I}R(g_i))\otimes_R\hm{G}{R}{M}{N}\ar[r]^(.48){\tau_{L,M,N}}\ar[d]_\cong&\hm{G}{R}{\hm{G}{R}{\bigoplus_{i\in I}R(g_i)}{M}}{N}\ar[d]^\cong\\\bigoplus_{i\in I}(R(g_i)\otimes_R\hm{G}{R}{M}{N})\ar[d]_\cong&\hm{G}{R}{\prod_{i\in I}\hm{G}{R}{R(g_i)}{M}}{N}\ar[d]^\cong\\\bigoplus_{i\in I}\hm{G}{R}{M}{N(g_i)})\ar@{ >->}[d]&\hm{G}{R}{\prod_{i\in I}M(-g_i)}{N}\ar@{ ->>}[d]\\\prod_{i\in I}\hm{G}{R}{M}{N(g_i)})\ar[r]^\cong&\hm{G}{R}{\bigoplus_{i\in I}M(-g_i)}{N}}\] in $\grmod^G(R)$ where the unmarked morphisms are the canonical ones (\ref{a40} A), \ref{a50} A)), implying that $\tau_{L,M,N}$ is a monomorphism. For the second claim we can by \ref{a40} A) suppose that $L=R$, and then it is clear.
\end{proof}

\begin{cor}\label{a130}
Let $M$ be a $G$-graded $R$-module. If $M$ is projective then $\tau_M$ is a monomorphism. If $M$ is projective and of finite type then $\tau_M$ is an isomorphism.
\end{cor}

\begin{proof}
Immediately from \ref{a120}.
\end{proof}

\begin{prop}\label{a140}
Let $M$ be a projective $G$-graded $R$-module of finite type.

a) If $N$ is a $G$-graded $R$-module that is small (or of finite type), then so is\linebreak $\hm{G}{R}{M}{N}$.

b) $\hm{G}{R}{M}{R}$ is projective and of finite type.
\end{prop}

\begin{proof}
a) Let $\LL=(L_j)_{j\in J}$ be a family of $G$-graded $R$-modules. There is a commutative diagram \[\xymatrix@C70pt{\bigoplus_{j\in J}(M\otimes_R\hm{G}{R}{N}{L_j})\ar[r]^(.48){\bigoplus_{j\in J}\tau_{M,N,L_j}}\ar[d]_\cong&\bigoplus_{j\in J}\hm{G}{R}{\hm{G}{R}{M}{N}}{L_j}\ar@{ >->}[dd]^{\Lambda_{\hm{G}{R}{M}{N},\LL}}\\M\otimes_R(\bigoplus_{j\in J}\hm{G}{R}{N}{L_j})\ar[d]_{M\otimes_R\Lambda_{N,\LL}}&\\M\otimes_R\hm{G}{R}{N}{\bigoplus_{j\in J}L_j}\ar[r]^(.47){\tau_{M,N,\bigoplus_{j\in J}L_j}}&\hm{G}{R}{\hm{G}{R}{M}{N}}{\bigoplus_{j\in J}L_j}}\] in $\grmod^G(R)$ where the unmarked morphism is the canonical one. If $M$ is projective and of finite type and $N$ is small, then both horizontal morphisms and $M\otimes_R\Lambda_{N,\LL}$ are isomorphisms (\ref{a120}), hence $\Lambda_{\hm{G}{R}{M}{N},\LL}$ is an isomorphism, too, and thus $\hm{G}{R}{M}{N}$ is small. If $N$ is additionally of finite type then there exists a finite family $(g_i)_{i\in I}$ in $G$ such that $M$ is a direct summand of $\bigoplus_{i\in I}R(g_i)$ (\ref{a50} B)), hence $\hm{G}{R}{M}{N}$ is a direct summand of $\hm{G}{R}{\bigoplus_{i\in I}R(g_i)}{N}\cong\bigoplus_{i\in I}N(-g_i)$ (\ref{a50} A)), and as this $G$-graded $R$-module is of finite type, the same holds for $\hm{G}{R}{M}{N}$.

b) As in a) we see that $\hm{G}{R}{M}{R}$ is a direct summand of the free $G$-graded $R$-module $\bigoplus_{i\in I}R(-g_i)$ and thus projective.
\end{proof}

\begin{no}
A) For $G$-graded $R$-modules $L$, $M$ and $N$ there is a commutative diagram \[\xymatrix@C70pt{(L\otimes_R\hm{G}{R}{M}{N})\ps\ar[r]^(.48){(\tau_{L,M,N})\ps}\ar[d]_\cong&\hm{G}{R}{\hm{G}{R}{L}{M}}{N}\ps\ar@{ >->}[d]^{\beta^\psi_{\hm{G}{R}{L}{M},N}}\\L\ps\otimes_{R\ps}\hm{G}{R}{M}{N}\ps\ar[d]_{L\ps\otimes_{R\ps}\beta^\psi_{M,N}}&\hm{H}{R\ps}{\hm{G}{R}{L}{M}\ps}{N\ps}\\L\ps\otimes_{R\ps}\hm{H}{R\ps}{M\ps}{N\ps}\ar[r]^(.46){\tau_{L\ps,M\ps,N\ps}}&\hm{H}{R\ps}{\hm{H}{R\ps}{L\ps}{M\ps}}{N\ps}.\ar[u]_{\hm{H}{R\ps}{\beta^\psi_{L,M}}{N\ps}}}\] in $\grmod^H(R\ps)$, where the unmarked morphism is the canonical one. If $\ke(\psi)$ is finite, or if $L$ is projective and of finite type and $M$ is small, then all the vertical morphisms are isomorphisms (\ref{a40} A), \ref{a50} A), \ref{a60} A), \ref{a140} a)).\smallskip

B) For a $G$-graded $R$-module $M$ there is a commutative diagram \[\xymatrix@C90pt{M\ps\ar[d]_{\tau_{M\ps}}\ar[r]^(.4){(\tau_M)\ps}&\hm{G}{R}{\hm{G}{R}{M}{R}}{R}\ps\ar@{ >->}[d]^{\beta^\psi_{\hm{G}{R}{M}{R},R}}\\\hm{H}{R\ps}{\hm{H}{R\ps}{M\ps}{R\ps}}{R\ps}\ar@{ >->}[r]^(.52){\hm{H}{R\ps}{\beta^\psi_{M,R}}{R\ps}}&\hm{H}{R\ps}{\hm{G}{R}{M}{R}\ps}{R\ps}}\] in $\grmod^H(R\ps)$. If $\ke(\psi)$ is finite, or if $M$ is projective and of finite type, then the right vertical and the lower horizontal morphism are isomorphisms (\ref{a50} A), \ref{a60} A), \ref{a140} a)).
\end{no}


\section{Change of ring functors}

\begin{no}\label{b10}
For a $G$-graded $S$-module $M$ we define a $G$-graded $R$-module $h_*(M)$ as follows: Its underlying additive group and its $G$-graduation are those of $M$; its $R$-scalar multiplication is given by $rx=h(r)x$ for $r\in R^\hme$ and $x\in M^\hme$, where the right side product is the $S$-scalar multiplication of $M$. If $u\colon M\rightarrow N$ is a morphism in $\grmod^G(S)$, then its underlying map defines a morphism $h_*(M)\rightarrow h_*(N)$ in $\grmod^G(R)$, denoted by $h_*(u)$. These definitions give rise to a faithful and conservative functor \[h_*\colon\grmod^G(S)\rightarrow\grmod^G(R),\] called \textit{scalar restriction (from $S$ to $R$) by means of $h$.} It is clear from the above that scalar restriction by means of $h$ commutes with $\psi$-coarsening, i.e., \[h_*(\bullet)\ps=(h\ps)_*(\bullet\ps).\]
\end{no}

\begin{no}\label{b20}
A) For a $G$-graded $S$-module $M$ and a $G$-graded $R$-module $N$ we define a $G$-graded $S$-module $M\otimes_RN$ as follows: Its underlying additive group and its $G$-graduation are those of $h_*(M)\otimes_RN$; its $S$-scalar multiplication is given by $s(x\otimes y)=(sx)\otimes y$ for $s\in S^\hme$, $x\in M^\hme$ and $y\in N^\hme$. If $u\colon M\rightarrow M'$ is a morphism in $\grmod^G(S)$ and $v\colon N\rightarrow N'$ is a morphism in $\grmod^G(R)$, then the map underlying \[h_*(u)\otimes v\colon h_*(M)\otimes_RN\rightarrow h_*(M')\otimes_RN'\] is $S$-linear and thus defines a morphism $M\otimes_RN\rightarrow M'\otimes_RN$ in $\grmod^G(S)$, denoted by $u\otimes v$. These definitions give rise to a bifunctor \[\bullet\otimes_R\sq\colon\grmod^G(S)\times\grmod^G(R)\rightarrow\grmod^G(S)\] with $h_*(\bullet\otimes_R\sq)=h_*(\bullet)\otimes_R\sq$. As $G$-graded tensor products commute with $\psi$-coarsening (\ref{a40} A)) it follows that the same holds for the above bifunctor, i.e., there is an isomorphism \[(\bullet\otimes_R\sq)\ps\cong(\bullet\ps)\otimes_{R\ps}(\sq\ps).\]

B) Taking $M=S$ in A) we get a functor \[h^*(\bullet)\dfgl S\otimes_R\bullet\colon\grmod^G(R)\rightarrow\grmod^G(S)\] with $h_*(h^*(\bullet))=h_*(S)\otimes_R\bullet$, called \textit{scalar extension (from $R$ to $S$) by means of $h$.} It follows from A) that $h^*$ commutes with $\psi$-coarsening, i.e., there is an isomorphism \[h^*(\bullet)\ps\cong(h\ps)^*(\bullet\ps).\] Moreover, there is a canonical isomorphism $h^*(R)\cong S$ in $\grmod^G(S)$.
\end{no}

\begin{no}\label{b30}
A) For a $G$-graded $S$-module $M$ and a $G$-graded $R$-module $N$ we define a $G$-graded $S$-module $\hm{G}{R}{M}{N}$ as follows: Its underlying additive group and its $G$-graduation are those of $\hm{G}{R}{h_*(M)}{N}$; its $S$-scalar multiplication is given by $(su)(x)=u(sx)$ for $s\in S^\hme$, $u\in\hm{G}{R}{h_*(M)}{N}^\hme$ and $x\in M^\hme$. If $u\colon M'\rightarrow M$ is a morphism in $\grmod^G(S)$ and $v\colon N\rightarrow N'$ is a morphism in $\grmod^G(R)$, then the map \[\hm{G}{R}{h_*(u)}{v}\colon\hm{G}{R}{h_*(M)}{N}\rightarrow\hm{G}{R}{h_*(M')}{N'}\] is $S$-linear and thus defines a morphism $\hm{G}{R}{M}{N}\rightarrow\hm{G}{R}{M'}{N'}$ in $\grmod^G(S)$, denoted by $\hm{G}{R}{u}{v}$. These definitions give rise to a bifunctor \[\hm{G}{R}{\bullet}{\sq}\colon\grmod^G(S)\times\grmod^G(R)\rightarrow\grmod^G(S)\] with $h_*(\hm{G}{R}{\bullet}{\sq})=\hm{G}{R}{h_*(\bullet)}{\sq}$. By \ref{a50} A) there is a canonical monomorphism of $H$-graded $R\ps$-modules \[\beta^\psi_{h_*(M),N}\colon\hm{G}{R}{h_*(M)}{N}\ps\rightarrowtail\hm{H}{R\ps}{h_*(M)\ps}{N\ps}.\] Its source equals $h_*(\hm{G}{R}{M}{N})\ps$, its target equals $(h\ps)_*(\hm{H}{R\ps}{M\ps}{N\ps})$, and on use of \ref{b10} it is readily checked that its underlying map is $S$-linear. Thus, it defines a monomorphism of $H$-graded $S\ps$-modules \[\beta^{\psi,h}_{M,N}\colon\hm{G}{R}{M}{N}\ps\rightarrowtail\hm{H}{R\ps}{M\ps}{N\ps}.\] As $\beta^\psi_{h_*(M),N}$ is natural in $M$ and $N$, the same holds for $\beta^{\psi,h}_{M,N}$, and so we get a canonical monomorphism of bifunctors \[\beta^{\psi,h}\colon\hm{G}{R}{\bullet}{\sq}\ps\rightarrowtail\hm{H}{R\ps}{\bullet\ps}{\sq\ps}.\] Clearly, $\beta^\psi=\beta^{\psi,\Id_R}$. It follows from \ref{a50} A) that $\beta^{\psi,h}_{M,N}$ is an isomorphism for every $G$-graded $R$-module $N$ if and only if $\ke(\psi)$ is finite or $h_*(M)$ is small.\smallskip

B) Taking $M=S$ in A) we get a functor \[\h(\bullet)\dfgl\hm{G}{R}{S}{\bullet}\colon\grmod^G(R)\rightarrow\grmod^G(S)\] with $h_*(\h(\bullet))=\hm{G}{R}{h_*(S)}{\bullet}$, called \textit{scalar coextension (from $R$ to $S$) by means of $h$.} By A) there is a canonical monomorphism \[\beta^{\psi,h}_{S,\bullet}\colon\h(\bullet)\ps\rightarrowtail\widetilde{h\ps}(\bullet\ps),\] and $\h$ commutes with $\psi$-coarsening, i.e., $\beta^{\psi,h}_{S,\bullet}$ is an isomorphism if and only if $\ke(\psi)$ is finite or $h_*(S)$ is small.
\end{no}

\begin{no}\label{b40}
Each of the functors $h_*$, $h^*$ and $\h$ commutes with shifts, i.e., for $g\in G$ there are isomorphisms of functors $h_*(\bullet(g))\cong h_*(\bullet)(g)$, $h^*(\bullet(g))\cong h^*(\bullet)(g)$, and $\h(\bullet(g))\cong\h(\bullet)(g)$.
\end{no}

\begin{no}\label{b50}
A) Let $M$ be a $G$-graded $S$-module. One can show analogously to the ungraded case (\cite[II.4.1 Proposition 1]{a}) that there are adjunctions \[\bigl(\grmod^G(S)\xrightarrow{h_*(M\otimes_S\bullet)}\grmod^G(R),\grmod^G(R)\xrightarrow{\hm{G}{R}{M}{\bullet}}\grmod^G(S)\bigr)\] and \[\bigl(\grmod^G(R)\xrightarrow{M\otimes_R\bullet}\grmod^G(S),\grmod^G(S)\xrightarrow{h_*(\hm{G}{S}{M}{\bullet})}\grmod^G(R)\bigr).\]

B) Let $L$ and $M$ be $G$-graded $S$-modules, and let $N$ be a $G$-graded $R$-module. By A), \ref{a50} A) and the symmetry of tensor products we have for $g\in G$ an isomorphism \[\hm{}{\grmod^G(R)}{h_*(L\otimes_SM)}{N(g)}\cong\hm{}{\grmod^G(S)}{L}{\hm{G}{R}{M}{N}(g)}\] in $\ab$. Taking the direct sum over $g\in G$ and keeping in mind \ref{b30} A) we get an isomorphism \[\alpha^h_{L,M,N}\colon\hm{G}{R}{L\otimes_SM}{N}\overset{\cong}\longrightarrow\hm{G}{S}{L}{\hm{G}{R}{M}{N}}\] in $\grmod^G(S)$ that is natural in $L$, $M$ and $N$. Moreover, there is a commutative diagram \[\xymatrix@C70pt{\hm{G}{R}{M\otimes_SL}{N}\ps\ar[r]^(.47){(\alpha^h_{M,L,N})\ps}_(.47)\cong\ar@{ >->}[dd]_{\beta^{\psi,h}_{M\otimes_SL,N}}&\hm{G}{S}{L}{\hm{G}{R}{M}{N}}\ps\ar@{ >->}[d]^{\beta^{\psi,h}_{L,\hm{G}{R}{M}{N}}}\\&\hm{H}{S\ps}{L\ps}{\hm{G}{R}{M}{N}\ps}\ar@{ >->}[d]^{\hm{H}{S\ps}{L\ps}{\beta^{\psi,h}_{M,N}}}\\\hm{H}{R\ps}{M\ps\otimes_{S\ps}L\ps}{N\ps}\ar[r]^(.46){\alpha^h_{M\ps,L\ps,N\ps}}_(.46)\cong&\hm{H}{S\ps}{L\ps}{\hm{H}{R\ps}{M\ps}{N\ps}}}\] in $\grmod^H(S\ps)$. If $\ke(\psi)$ is finite or $h_*(L)$ and $h_*(M)$ are small, then all the vertical monomorphisms are isomorphisms (\ref{a60} C), \ref{b30} A)).
\end{no}

\begin{no}\label{b60}
A) Let $L$ and $N$ be $G$-graded $S$-modules, and let $M$ be a $G$-graded $R$-module. There is a morphism\[\pi^h_{L,M,N}\colon\hm{G}{R}{L}{M}\otimes_SN\rightarrow\hm{G}{S}{L}{M\otimes_RN}\] in $\grmod^G(S)$ with $u\otimes x\mapsto(y\mapsto u(y)\otimes x)$ for $u\in\hm{G}{R}{L}{M}^\hme$, $x\in N^\hme$ and $y\in L^\hme$ that is natural in $L$, $M$ and $N$. Moreover, there is a commutative diagram \[\xymatrix@C70pt{(\hm{G}{R}{L}{M}\otimes_SN)\ps\ar[r]^{(\pi^h_{L,M,N})\ps}\ar[d]_{\beta^{\psi,h}_{L,M}\otimes_{S\ps}N\ps}&\hm{G}{S}{L}{M\otimes_RN}\ps\ar@{ >->}[d]^{\beta^{\psi,h}_{L,M\otimes_RN}}\\\hm{H}{R\ps}{L\ps}{M\ps}\otimes_{S\ps}N\ps\ar[r]^(.49){\pi^{h\ps}_{L\ps,M\ps,N\ps}}&\hm{H}{S\ps}{L\ps}{M\ps\otimes_{R\ps}N\ps}}\] in $\grmod^H(S\ps)$. If $\ke(\psi)$ is finite or $h_*(L)$ is small, then both vertical morphisms are isomorphisms (\ref{b30} A)).\smallskip

B) Let $L$ be a $G$-graded $S$-module, and let $M$ and $N$ be $G$-graded $R$-modules. By A) and \ref{b30} A) there is a morphism \[\pi^{\Id_R}_{h_*(L),M,N}\colon h_*(\hm{G}{R}{L}{M})\otimes_RN\rightarrow h_*(\hm{G}{R}{L}{M\otimes_RN})\] in $\grmod^G(R)$ that is $S$-linear. Thus, there is a morphism \[\nu^h_{L,M,N}\colon\hm{G}{R}{L}{M}\otimes_RN\rightarrow\hm{G}{R}{L}{M\otimes_RN}\] in $\grmod^G(S)$ with $h_*(\nu^h_{L,M,N})=\pi^{\Id_R}_{h_*(L),M,N}$ that is natural in $L$, $M$ and $N$. Moreover, there is a commutative diagram \[\xymatrix@C70pt{(\hm{G}{R}{L}{M}\otimes_RN)\ps\ar[r]^{(\nu^h_{L,M,N})\ps}\ar[d]_{\beta^{\psi,h}_{L,M}\otimes_{R\ps}N\ps}&\hm{G}{R}{L}{M\otimes_RN}\ps\ar@{ >->}[d]^{\beta^{\psi,h}_{L,M\otimes_RN}}\\\hm{H}{R\ps}{L\ps}{M\ps}\otimes_{R\ps}N\ps\ar[r]^(.49){\nu^h_{L\ps,M\ps,N\ps}}&\hm{H}{R\ps}{L\ps}{M\ps\otimes_{R\ps}N\ps}}\] in $\grmod^H(S\ps)$. If $\ke(\psi)$ is finite or $h_*(L)$ is small, then both vertical morphisms are isomorphisms (\ref{b30} A)).\smallskip

C) Let $M$ be a $G$-graded $S$-module and let $N$ be a $G$-graded $R$-module. There is a morphism \[\mu^h_{M,N}\colon M\otimes_RN\rightarrow\hm{G}{R}{\hm{G}{R}{M}{R}}{N}\] in $\grmod^G(S)$ with $x\otimes y\mapsto(u\mapsto u(x)y)$ for $x\in M^\hme$, $y\in N^\hme$ and $u\in\hm{G}{R}{M}{R}^\hme$ that is natural in $M$ and $N$. Moreover, there is a commutative diagram \[\xymatrix@C90pt{(M\otimes_RN)\ps\ar[r]^(.42){(\mu^h_{M,N})\ps}\ar[d]_{\mu^{h\ps}_{M\ps,N\ps}}&\hm{G}{R}{\hm{G}{R}{M}{R}}{N}\ps\ar@{ >->}[d]^{\beta^{\psi,h}_{\hm{G}{R}{M}{R},N}}\\\hm{H}{R\ps}{\hm{H}{R\ps}{M\ps}{R\ps}}{N\ps}\ar@{ >->}[r]^(.52){\hm{H}{R\ps}{\beta^{\psi,h}_{M,R}}{N\ps}}&\hm{H}{R\ps}{\hm{G}{R}{M}{R}\ps}{N\ps}}\] in $\grmod^H(S\ps)$. If $\ke(\psi)$ is finite or $h_*(M)$ is projective and of finite type, then the lower horizontal and the right vertical morphism are isomorphisms (\ref{a140} a), \ref{b30} A)).\smallskip

D) For $G$-graded $R$-modules $M$ and $N$ there is a commutative diagram \[\xymatrix{M\otimes_RN\ar[rr]^(.45){\mu^{\Id_R}_{M,N}}\ar[rd]_(.4){\tau_M\otimes_RN\quad}&&\hm{G}{R}{\hm{G}{R}{M}{R}}{N}\\&\hm{G}{R}{\hm{G}{R}{M}{R}}{R}\otimes_RN\ar[ru]_(.6){\qquad\nu^{\Id_R}_{\hm{G}{R}{M}{R},R,N}}}\] in $\grmod^G(R)$.
\end{no}

\begin{prop}\label{b70}
Let $L$ be a $G$-graded $S$-module, and let $M$ and $N$ be $G$-graded $R$-modules. If $N$ is projective then $\nu^h_{L,M,N}$ is a monomorphism. If $h_*(L)$ or $N$ is projective and of finite type, or if $h_*(L)$ is small and $N$ is projective, then $\nu^h_{L,M,N}$ is an isomorphism.\footnote{For $G=0$ the last statement generalises \cite[II.4 Exercice 3]{a}.}
\end{prop}

\begin{proof}
Let $L'$ be a further $G$-graded $S$-module and let $N'$ be a further $G$-graded $R$-module. Then, there are commutative diagrams \[\xymatrix{\hm{G}{R}{L\oplus L'}{M}\otimes_RN\ar[d]_{\nu^h_{L\oplus L',M,N}}\ar[r]^(.37){\cong}&(\hm{G}{R}{L}{M}\otimes_RN)\oplus(\hm{G}{R}{L'}{M}\otimes_RN)\ar[d]^{\nu^h_{L,M,N}\oplus\nu^h_{L',M,N}}\\\hm{G}{R}{L\oplus L'}{M\otimes_RN}\ar[r]^(.39){\cong}&\hm{G}{R}{L}{M\otimes_RN}\oplus\hm{G}{R}{L'}{M\otimes_RN}}\] and \[\xymatrix{\hm{G}{R}{L}{M}\otimes_R(N\oplus N')\ar[d]_{\nu^h_{L,M,N\oplus N'}}\ar[r]^(.38){\cong}&(\hm{G}{R}{L}{M}\otimes_RN)\oplus(\hm{G}{R}{L}{M}\otimes_RN')\ar[d]^{\nu^h_{L,M,N}\oplus\nu^h_{L,M,N'}}\\\hm{G}{R}{L}{M\otimes_R(N\oplus N')}\ar[r]^(.4){\cong}&\hm{G}{R}{L}{M\otimes_RN}\oplus\hm{G}{R}{L}{M\otimes_RN'}}\] in $\grmod^G(S)$ where the unmarked morphisms are the canonical ones. It follows that $\nu^h_{L\oplus L',M,N}$ is a mono- or isomorphism if and only if both $\nu^h_{L,M,N}$ and $\nu^h_{L',M,N}$ are so, and that $\nu^h_{L,M,N\oplus N'}$ is a mono- or isomorphism if and only if both $\nu^h_{L,M,N}$ and $\nu^h_{L,M,N'}$ are so.

If $N$ is projective, then $N_{[0]}$ is projective and hence flat (\ref{a40} B), \ref{a50} B)), so that $\beta^0_{L,M}\otimes_{R_{[0]}}N_{[0]}$ is a monomorphism (\ref{b30} A)) and that $\nu^h_{L_{[0]},M_{[0]},N_{[0]}}$ is a monomorphism by the corresponding ungraded statement (\cite[II.4.2 Proposition 2]{a}). Now, $\nu^h_{L,M,N}$ is a monomorphism by \ref{b60} B).

If $h_*(L)$ or $N$ is projective and of finite type, then by the first paragraph we can suppose first that it is free of finite rank, and second that it equals $R$, in which case the claim is clear. So, suppose that $h_*(L)$ is small and $N$ is projective. By the first paragraph we can suppose that $N=\bigoplus_{i\in I}R(g_i)$ for a family $(g_i)_{i\in I}$ in $G$. Then, there is a commutative diagram \[\xymatrix@C40pt@R15pt{\hm{G}{R}{h_*(L)}{M}\otimes_R(\bigoplus_{i\in I}R(g_i))\ar[r]^{\nu^{\Id_R}_{h_*(L),M,N}}\ar[d]_\cong&\hm{G}{R}{h_*(L)}{M\otimes_R(\bigoplus_{i\in I}R(g_i))}\ar[d]^\cong\\\bigoplus_{i\in I}\hm{G}{R}{h_*(L)}{M(g_i)}\ar@{ >->}[r]^{\Lambda_{h_*(L),(M(g_i))_{i\in I}}}_\cong&\hm{G}{R}{h_*(L)}{\bigoplus_{i\in I}M(g_i)}}\] in $\grmod^G(R)$, where the vertical morphisms are the canonical ones (\ref{a40} A), \ref{a50} A), \ref{a60} A)). The claim follows now since $\nu^{\Id_R}_{h_*(L),M,N}=h_*(\nu^h_{L,M,N})$ and $h_*$ is conservative (\ref{b10}).
\end{proof}

\begin{cor}\label{b80}
Let $M$ and $N$ be $G$-graded $R$-modules. If $M$ and $N$ are projective then $\mu^{\Id_R}_{M,N}$ is a monomorphism. If $M$ is projective and of finite type then $\mu^{\Id_R}_{M,N}$ is an isomorphism.
\end{cor}

\begin{proof}
The first claim follows from \ref{a50} B), \ref{a130}, \ref{b60} D) and \ref{b70}. The second claim follows from \ref{a130}, \ref{a140} b), \ref{b60} D) and \ref{b70}.
\end{proof}


\section{Adjunctions}

\begin{no}\label{c10}
A) For a $G$-graded $R$-module $M$ there are morphisms \[\rho^h_M\colon M\rightarrow h_*(h^*(M)),\;x\mapsto 1_S\otimes x\] and \[\widetilde{\sigma}^h_M\colon h_*(\h(M))\rightarrow M,\;u\mapsto u(1_S)\] in $\grmod^G(R)$ that are natural in $M$. For a $G$-graded $S$-module $N$ there are morphisms \[\sigma^h_N\colon h^*(h_*(N))\rightarrow N,\;\sum_{i\in I}s_i\otimes x_i\mapsto\sum_{i\in I}s_ix_i\] and \[\widetilde{\rho}^h_N\colon N\rightarrow\h(h_*(N)),\;x\mapsto h_*(s\mapsto sx)\] in $\grmod^G(S)$ that are natural in $N$. Altogether we have morphisms of functors \[\rho^h\colon\Id_{\grmod^G(R)}\rightarrow h_*\circ h^*,\quad\sigma^h\colon h^*\circ h_*\rightarrow\Id_{\grmod^G(S)},\]\[\widetilde\rho^h\colon\Id_{\grmod^G(S)}\rightarrow\h\circ h_*,\quad\widetilde\sigma^h\colon h_*\circ\h\rightarrow\Id_{\grmod^G(R).}\] Moreover, if $g\in G$ then $\rho^h_M(g)=\rho^h_{M(g)}$, $\sigma^h_N(g)=\sigma^h_{N(g)}$, $\widetilde{\rho}^h_N(g)=\widetilde{\rho}^h_{N(g)}$, and $\widetilde{\sigma}^h_M(g)=\widetilde{\sigma}^h_{M(g)}$.\smallskip

B) By \ref{b10} and \ref{b20} B) there are isomorphisms of functors \[h_*(h^*(\bullet))\ps\cong(h\ps)_*((h\ps)^*(\bullet\ps))\quad\text{and}\quad h^*(h_*(\bullet))\ps\cong(h\ps)^*((h\ps)_*(\bullet\ps))\] such that the diagrams of functors \[\xymatrix{h^*(h_*(\bullet))\ps\ar[rr]^(.55){\bullet\ps\circ\sigma^h}\ar[d]_\cong&&\bullet\ps&\bullet\ps\ar[rr]^(.43){\bullet\ps\circ\rho^h}\ar@/_1pc/[rrd]_(.45){\rho^{h\ps}\circ\bullet\ps\quad}&&h_*(h^*(\bullet))\ps\ar[d]^{\cong}\\(h\ps)^*((h\ps)_*(\bullet\ps))\ar@/_1pc/[rru]_(.55){\quad\sigma^{h\ps}\circ\bullet\ps}&&&&&(h\ps)_*((h\ps)^*(\bullet\ps))}\] commute.\smallskip

C) By \ref{b10} and \ref{b30} B) there are monomorphisms of functors \[h_*(\h(\bullet))\ps\rightarrowtail(h\ps)_*(\widetilde{h\ps}(\bullet\ps))\quad\text{and}\quad\h(h_*(\bullet))\ps\rightarrowtail\widetilde{h\ps}((h\ps)_*(\bullet\ps))\] such that the diagrams of functors \[\xymatrix{
h_*(\h(\bullet))\ps\ar@{ >->}[d]\ar[rr]^(.55){\bullet\ps\circ\widetilde{\sigma}^h}&&
\bullet\ps&\bullet\ps\ar[rr]^(.43){\bullet\ps\circ\widetilde{\rho}^h}\ar@/_1pc/[rrd]_(.45){\widetilde{\rho}^{h\ps}\circ\bullet\ps}&&\h(h_*(\bullet))\ps\ar@{ >->}[d]\\(h\ps)_*(\widetilde{h\ps}(\bullet\ps))\ar@/_1pc/[rru]_(.55){\quad\widetilde{\sigma}^{h\ps}\circ\bullet\ps}&&&&&\widetilde{h\ps}((h\ps)_*(\bullet\ps))}\] commute. Each of these monomorphisms is an isomorphism if and only if $\ke(\psi)$ is finite or $h_*(S)$ is small.
\end{no}

\begin{thm}\label{c20}
There are an adjunction \[\bigl(\grmod^G(R)\overset{h^*}\longrightarrow\grmod^G(S),\grmod^G(S)\overset{h_*}\longrightarrow\grmod^G(R)\bigr)\] with unit $\rho^h$ and counit $\sigma^h$, and an adjunction \[\bigl(\grmod^G(S)\overset{h_*}\longrightarrow\grmod^G(R),\grmod^G(R)\overset{\h}\longrightarrow\grmod^G(S)\bigr)\] with unit $\widetilde{\rho}^h$ and counit $\widetilde{\sigma}^h$.
\end{thm}

\begin{proof}
This is readily checked on use of \ref{a05} A).
\end{proof}

\begin{cor}\label{c25}
There are isomorphisms \[h_*(\hm{G}{S}{h^*(\bullet)}{\sq})\cong\hm{G}{R}{\bullet}{h_*(\sq)}\quad\text{and}\quad\hm{G}{R}{h_*(\bullet)}{\sq}\cong h_*(\hm{G}{S}{\bullet}{\h(\sq)}).\]
\end{cor}

\begin{proof}
This is readily checked on use of \ref{b40} and \ref{c20}.
\end{proof}

\begin{cor}
If $l\colon S\rightarrow T$ is a further morphism of $G$-graded rings, then $(l\circ h)_*=h_*\circ l_*$, and there are isomorphisms of functors $(l\circ h)^*\cong l^*\circ h^*$ and $\widetilde{l\circ h}\cong\widetilde{l}\circ\h$.
\end{cor}

\begin{proof}
The first claim follows immediately from the definition of $h_*$ (\ref{b10}). Together with \ref{c20} and \ref{a05} B), C) it implies the other claims.
\end{proof}

\begin{cor}\label{c40}
$\sigma^h$ is an epimorphism and $\widetilde\rho^h$ is a monomorphism.
\end{cor}

\begin{proof}
Since $h_*$ is faithful (\ref{b10}), this follows immediately from \ref{a05} E) and \ref{c20}.
\end{proof}

\begin{prop}\label{c50}
a) $\rho^h$ is a mono-, epi- or isomorphism if and only if $\hsub$ is pure, an epimorphism, or an isomorphism.

b) $\widetilde\sigma^h$ is a mono-, epi- or isomorphism, resp., if and only if $\hsub$ is an epimorphism, a section, or an isomorphism, resp.
\end{prop}

\begin{proof}
a) We have a commutative diagram \[\xymatrix@C40pt@R15pt{\Id_{\grmod^G(R)}(\bullet)\ar[r]^(.43){\rho^h}\ar[d]_{\cong}&h_*(h^*(\bullet))\ar@{=}[d]\\R\otimes_R\bullet\ar[r]^(.45){\hsub\otimes_R\bullet}&h_*(S)\otimes_R\bullet}\] of functors, where the left vertical morphism is the canonical one (\ref{b20} A)). Therefore, $\rho^h$ is a mono-, epi- or isomorphism if and only if $\hsub\otimes_R\bullet$ is so, thus if and only if $\hsub$ is pure, an epi-, or an isomorphism.

b) We have a commutative diagram \[\xymatrix@C60pt@R15pt{h_*(\h(\bullet))\ar[r]^{\widetilde{\sigma}^h}\ar@{=}[d]&\Id_{\grmod^G(R)}(\bullet)\\\hm{G}{R}{h_*(S)}{\bullet}\ar[r]^(.53){\hm{G}{R}{\hsub}{\bullet}}&\hm{G}{R}{R}{\bullet}\ar[u]_\cong}\] of functors, where the right vertical morphism is the canonical one. Therefore, $\widetilde{\sigma}^h$ is a mono-, epi- or isomorphism if and only if $\hm{G}{R}{\hsub}{\bullet}$ is so, hence if and only if $\hm{}{\grmod^G(R)}{\hsub}{\bullet}$ is a mono-, epi- or isomorphism, and thus if and only if $\hsub$ is an epimorphism, a section, or an isomorphism.
\end{proof}

\begin{cor}
a) $\rho^h$ is a mono-, epi- or isomorphism if and only if $\rho^{h\ps}$ is so.

b) $\widetilde{\sigma}^h$ is a mono-, epi- or isomorphism if and only if $\widetilde{\sigma}^{h\ps}$ is so.
\end{cor}

\begin{proof}
This follows from \ref{a10}, \ref{a20}, \ref{a40} C), and \ref{c50}.
\end{proof}

\begin{no}\label{c70}
A) If $M$ is a small $G$-graded $R$-module then $h^*(M)$ is small by \ref{a60} A) and the corresponding ungraded result (\cite[3$^\circ$]{rentschler}). The converse need not hold; an (ungraded) counterexample is given by the canonical bimorphism $h\colon\Z\hookrightarrow\Q$ and the non-small $\Z$-module $M=(\Z/2\Z)^{\oplus\N}$, for $h^*(M)=0$ is small.\smallskip

B) If $N$ is a $G$-graded $S$-module such that $h_*(N)$ is small, then $N$ is small. Indeed, by \ref{a60} A) we can suppose that $G=0$. If $h_*(N)$ is small, then so is $h^*(h_*(N))$ by A), and since $\sigma^h_N\colon h^*(h_*(N))\rightarrow N$ is an epimorphism (\ref{c40}) it follows from \cite[2$^\circ$]{rentschler} that $N$ is small. The converse need not hold; an (ungraded) counterexample is given by the canonical bimorphism $h\colon\Z\hookrightarrow\Q$ and the small $\Q$-module $\Q$, for $h_*(\Q)$ is not small.\smallskip

C) If $h_*(S)$ is projective and of finite type and $M$ is a small $G$-graded $R$-module, then $\h(M)$ is small by \ref{a140} a) and B).
\end{no}

\begin{prop}\label{c79}
a) $h_*$ commutes with inductive and projective limits, $h^*$ commutes with inductive limits, and $\h$ commutes with projective limits.

b) $h^*$ is exact if and only if $h_*(S)$ is flat, commutes with products if and only if $h_*(S)$ is of finite presentation, and commutes with projective limits if and only if $h_*(S)$ is projective and of finite type.

c) $\h$ is exact if and only if $h_*(S)$ is projective, commutes with direct sums if and only if $h_*(S)$ is small, and commutes with inductive limits if and only if $h_*(S)$ is projective and of finite type.
\end{prop}

\begin{proof}
a) follows from \ref{a05} D) and \ref{c20}.

b) $h^*$ is exact if and only if $h_*(S)\otimes_R\bullet$ is so (\ref{b10}, \ref{b20} A), a)), hence if and only if $h_*(S)$ is flat. It commutes with products if and only if $h_*(S)\otimes_R\bullet$ does so (\ref{b20} A), a)), and thus the second claim follows from \ref{a90} b). The last claim follows from b), c), \ref{a05} F) and \ref{a50} B).

c) $\h$ is exact if and only if $\hm{G}{R}{h_*(S)}{\bullet}$ is so (\ref{b10}, \ref{b30} A), a)), hence if and only if $h_*(S)$ is projective (\ref{a50} B)). It commutes with direct sums if and only if $\hm{G}{R}{h_*(S)}{\bullet}$ does so (\ref{b30} A), a)), hence if and only if $h_*(S)$ is small (\ref{a60} A)). The last claim follows from b), c), \ref{a05} F) and \ref{a60} A).
\end{proof}

\begin{cor}
a) $h^*$ is exact, commutes with products, or commutes with projective limits if and only if $(h\ps)^*$ has the same property.

b) $\h$ is exact, commutes with direct sums, or commutes with inductive limits if and only if $\widetilde{h\ps}$ has the same property.
\end{cor}

\begin{proof}
Since $h_*$ commutes with $\psi$-coarsening (\ref{b10}), this follows from \ref{a30}, \ref{a40} B), \ref{a50} B), \ref{a60} A), and \ref{c79}.
\end{proof}

\begin{cor}\label{c120}
The following statements are equivalent: (i) $h^*$ has a left adjoint; (ii) $\h$ has a right adjoint; (iii) The $G$-graded $R$-module $h_*(S)$ is projective and of finite type.
\end{cor}

\begin{proof}
Immediately from \ref{a05} D) and \ref{c79}.
\end{proof}

\begin{no}\label{c130}
We define functors \[h_+(\bullet)\dfgl h_*(\bullet\otimes_S\h(R))\colon\grmod^G(S)\rightarrow\grmod^G(R)\] and \[h^+(\bullet)\dfgl h_*(\hm{G}{S}{\h(R)}{\bullet})\colon\grmod^G(S)\rightarrow\grmod^G(R).\] By \ref{b50} A), $h_+$ has a right adjoint, namely $\hm{G}{R}{\h(R)}{\bullet}$, and $h^+$ has a left adjoint, namely $\h(R)\otimes_R\bullet$.
\end{no}

\begin{prop}\label{c140}
If $h_*(S)$ is projective and of finite type, then there are adjunctions \[\bigl(\grmod^G(S)\overset{h_+}\longrightarrow\grmod^G(R),\grmod^G(R)\overset{h^*}\longrightarrow\grmod^G(S)\bigr)\] and \[\bigl(\grmod^G(R)\overset{\h}\longrightarrow\grmod^G(S),\grmod^G(S)\overset{h^+}\longrightarrow\grmod^G(R)\bigr).\]
\end{prop}

\begin{proof}
For a $G$-graded $R$-module $M$, the isomorphisms $h_*(\mu^h_{S,M})=\mu^{\Id_R}_{h_*(S),M}$ in $\grmod^G(R)$ and $\nu^h_{S,R,M}$ in $\grmod^G(S)$ (\ref{b70}, \ref{b80}) give rise to isomorphisms \[h^*(M)\overset{\cong}\longrightarrow\hm{G}{R}{\h(R)}{M}\quad\text{and}\quad\h(R)\otimes_RM\overset{\cong}\longrightarrow\h(M)\] in $\grmod^G(S)$ that are natural in $M$. Thus, \ref{c130} yields the desired adjunctions.
\end{proof}

\begin{cor}\label{c150}
There is an isomorphism $h^*\cong\h$ if and only if $h_*(S)$ is projective and of finite type and $\h(R)\cong S$.\footnote{For $G=0$ this is contained in \cite[Theorem 4.1]{morita}.}
\end{cor}

\begin{proof}
Necessity follows from \ref{c20} and \ref{c120} since $h^*(R)\cong S$ (\ref{b20} B)). Sufficiency follows from \ref{c130} and \ref{c140}. 
\end{proof}


\section{Interplay with tensor and Hom, and epimorphisms}

\begin{no}\label{d10}
A) Let $M$ and $N$ be $G$-graded $R$-modules. There is a morphism \[\delta^h_{M,N}\colon h^*(M)\otimes_Sh^*(N)\rightarrow h^*(M\otimes_RN)\] in $\grmod^G(S)$ with $(r\otimes x)\otimes(s\otimes y)\mapsto(rs)\otimes(x\otimes y)$ for $r,s\in S^\hme$, $x\in M^\hme$ and $y\in N^\hme$, and this is natural in $M$ and $N$. Moreover, there is a morphism \[\overline{\delta}^h_{M,N}\colon h^*(M\otimes_RN)\rightarrow h^*(M)\otimes_Sh^*(N)\] in $\grmod^G(S)$ with $s\otimes(x\otimes y)\mapsto s((1_S\otimes x)\otimes(1_S\otimes y))$ for $s\in S^\hme$, $x\in M^\hme$ and $y\in N^\hme$. As $\delta^h_{M,N}$ and $\overline{\delta}^h_{M,N}$ are mutually inverse, we have an isomorphism \[\delta^h\colon h^*(\bullet)\otimes_Sh^*(\sq)\overset{\cong}\longrightarrow h^*(\bullet\otimes_R\sq).\]

B) Since $\psi$-coarsening commutes with tensor products and scalar extension (\ref{a40} A), \ref{b20} B)), we have a commutative diagram \[\xymatrix@C60pt@R15pt{(h^*(M)\otimes_Sh^*(N))\ps\ar[r]_\cong^{(\delta^h_{M,N})\ps}\ar[d]_\cong&h^*(M\otimes_RN)\ps\ar[d]^\cong\\(h\ps)^*(M\ps)\otimes_{S\ps}(h\ps)^*(N\ps)\ar[r]_(.55)\cong^(.54){\delta^{h\ps}_{M\ps,N\ps}}&(h\ps)^*(M\ps\otimes_{R\ps}N\ps)}\] in $\grmod^H(S\ps)$, where the vertical morphisms are the canonical ones.
\end{no}

\begin{no}\label{d20}
A) Let $M$ and $N$ be $G$-graded $S$-modules. There is an epimorphism \[\gamma^h_{M,N}\colon h_*(M)\otimes_Rh_*(N)\twoheadrightarrow h_*(M\otimes_SN)\] in $\grmod^G(R)$ with $x\otimes y\mapsto x\otimes y$ for $x\in h_*(M)^\hme$ and $y\in h_*(N)^\hme$, and this is natural in $M$ and $N$. Therefore, we have an epimorphism \[\gamma^h\colon h_*(\bullet)\otimes_Rh_*(\sq)\twoheadrightarrow h_*(\bullet\otimes_S\sq).\] Furthermore, we have $\gamma^h_{S,N}=h_*(\sigma^h_N)$ (\ref{c10} A)).\smallskip

B) Since $\psi$-coarsening commutes with tensor products and scalar restriction (\ref{a40} A), \ref{b10}), we have a commutative diagram \[\xymatrix@C60pt@R15pt{(h_*(M)\otimes_Rh_*(N))\ps\ar@{->>}[r]^(.53){(\gamma^h_{M,N})\ps}\ar[d]_\cong&h_*(M\otimes_SN)\ps\ar[d]^\cong\\(h\ps)_*(M\ps)\otimes_{R\ps}(h\ps)_*(N\ps)\ar@{->>}[r]^(.54){\gamma^{h\ps}_{M\ps,N\ps}}&(h\ps)_*(M\ps\otimes_{S\ps}N\ps)}\] in $\grmod^H(R\ps)$, where the vertical morphisms are the canonical ones.
\end{no}

\begin{no}\label{d25}
A) Let $M$ and $N$ be $G$-graded $R$-modules. There is a morphism \[\eps^h_{M,N}\colon\h(M)\otimes_S\h(N)\rightarrow\h(M\otimes_RN)\] in $\grmod^G(S)$ with $u\otimes v\mapsto(s\mapsto u(s)\otimes v(1_S))$ for $u\in\h(M)^\hme$, $v\in\h(N)^\hme$ and $s\in S^\hme$, and this is natural in $M$ and $N$. Therefore, we have a morphism \[\eps^h\colon\h(\bullet)\otimes_S\h(\sq)\rightarrow\h(\bullet\otimes_R\sq).\]

B) Since $\psi$-coarsening commutes with tensor products (\ref{a40} A)) it follows from \ref{b30} A) that we have a commutative diagram \[\xymatrix@C50pt@R15pt{(\h(M)\otimes_S\h(N))\ps\ar[r]^(.53){(\eps^h_{M,N})\ps}\ar@{ >->}[d]&\h(M\otimes_RN)\ps\ar@{ >->}[d]\\\widetilde{h\ps}(M\ps)\otimes_{S\ps}\widetilde{h\ps}(N\ps)\ar[r]^(.53){\eps^{h\ps}_{M\ps,N\ps}}&\widetilde{h\ps}(M\ps\otimes_{R\ps}N\ps)}\] in $\grmod^H(S\ps)$, where the vertical morphisms are induced by $\beta^{\psi,h}_{S,M}\otimes_{S\ps}\beta^{\psi,h}_{S,N}$ and\linebreak $\beta^{\psi,h}_{S,M\otimes_RN}$. If $\ke(\psi)$ is finite or $h_*(S)$ is small, then the vertical morphisms are isomorphisms.\smallskip

C) On use of the symmetry of tensor products and the canonical isomorphism\linebreak $\hm{G}{S}{S}{\bullet}\cong\Id_{\grmod^G(S)}$ it is readily checked that there is a commutative diagram \[\xymatrix{\h(M)\otimes_S\h(N)\ar[rr]^{\eps^h_{M,N}}\ar[rd]_{\pi^h_{S,M,\h(N)}}&&\h(M\otimes_RN)\\&M\otimes_R\h(N)\ar[ru]_{\nu^h_{S,N,M}}&}\] in $\grmod^G(S)$.\smallskip

D) The morphism $\eps^h$ need not be an epimorphism. Indeed, let $R$ be a finite nonzero ring, and let $h\colon R\rightarrow R[X]$ be the polynomial algebra in one indeterminate over $R$. The $R$-module $R^{\oplus\N}$ is countable but not of finite type. The first part of the proof of \ref{a90} a) shows that the canonical morphism of $R$-modules $\kappa_{R^{\oplus\N},(R)_{i\in\N}}\colon(R^\N)\otimes_R(R^{\oplus\N})\rightarrow (R^{\oplus\N})^\N$ (\ref{a70} A)) is not an epimorphism. But up to the isomorphism $\hm{G}{R}{R[X]}{\bullet}\cong\bullet^\N$, this morphism equals $\nu^h_{R[X],R,R^{\oplus\N}}$, and thus C) implies that $\eps^h_{R^{\oplus\N},R}$ is not an epimorphism.\smallskip

E) The morphism $\eps^h$ need not be a monomorphism. Indeed, let $K$ be a field, let $R\dfgl K[X]/\langle X^3\rangle$, let $Y$ denote the canonical image of $X$ in $Y$, let $\ia\dfgl\langle Y^2\rangle_R$, and let $h\colon R\twoheadrightarrow R/\ia$ be the canonical projection. Then, $\pi^h_{R/\ia,R,\h(R/\ia)}$ takes the form $(0:_R\ia)\rightarrow R/\ia,\;x\mapsto x+\ia$, hence maps $Y\neq 0$ to $0$, and thus is not a monomorphism. Therefore, C) implies that $\eps^h_{R,R/\ia}$ is not a monomorphism either.
\end{no}

\begin{no}\label{d30}
A) Let $M$ and $N$ be $G$-graded $S$-modules. There is a monomorphism of $G$-graded $R$-modules \[\eta^h_{M,N}\colon h_*(\hm{G}{S}{M}{N})\rightarrowtail\hm{G}{R}{h_*(M)}{h_*(N)},\;u\mapsto h_*(u)\] in $\grmod^G(R)$, and this is natural in $M$ and $N$. Therefore, we have a monomorphism \[\eta^h\colon h_*(\hm{G}{S}{\bullet}{\sq})\rightarrowtail\hm{G}{R}{h_*(\bullet)}{h_*(\sq)}.\] Furthermore, identifying $N$ and $\hm{G}{S}{S}{N}$ by means of the canonical isomorphism we have $\eta^h_{S,N}=h_*(\widetilde\rho^h_N)$ (\ref{c10} A)).\footnote{Since $h_*$ is faithful (\ref{b10}) and $\eta^h$ is a monomorphism, this implies again that $\widetilde\rho^h$ is a monomorphism (\ref{c40}).}\smallskip

B) Since $\psi$-coarsening commutes with scalar restriction (\ref{b10}) it follows from \ref{a50} A) and \ref{c79} a) that we have a commutative diagram \[\xymatrix@C60pt{h_*(\hm{G}{S}{M}{N})\ps\ar@{ >->}[r]^{(\eta^h_{M,N})\ps}\ar@{ >->}[d]_{(h\ps)_*(\beta^\psi_{M,N})}&\hm{G}{R}{h_*(M)}{h_*(N)}\ps\ar@{ >->}[d]^{\beta^\psi_{h_*(M),h_*(N)}}\\(h\ps)_*(\hm{H}{S\ps}{M\ps}{N\ps})\ar@{ >->}[r]^(.45){\eta^{h\ps}_{M\ps,N\ps}}&\hm{H}{R\ps}{(h\ps)_*(M\ps)}{(h\ps)_*(N\ps)}}\] in $\grmod^H(R\ps)$. If $\ke(\psi)$ is finite or $h_*(M)$ is small, then both vertical morphisms are isomorphisms (\ref{a50} A), \ref{c70} B)).
\end{no}

\begin{no}\label{d40}
A) Let $M$ and $N$ be $G$-graded $R$-modules. There is a morphism \[\thet^h_{M,N}\colon h^*(\hm{G}{R}{M}{N})\rightarrow\hm{G}{S}{h^*(M)}{h^*(N)}\] in $\grmod^G(S)$ with $s\otimes u\mapsto(r\otimes x\mapsto(rs)\otimes u(x))$ for $r,s\in S^\hme$, $u\in\hm{G}{R}{M}{N}^\hme$ and $x\in M^\hme$, and this is natural in $M$ and $N$. Therefore, we have a morphism \[\thet^h\colon h^*(\hm{G}{R}{\bullet}{\sq})\rightarrow\hm{G}{S}{h^*(\bullet)}{h^*(\sq)}.\]

B) Since $\psi$-coarsening commutes with scalar extension (\ref{b20} B)) it follows from \ref{a50} A) that we have a commutative diagram \[\xymatrix@C60pt{h^*(\hm{G}{R}{M}{N})\ps\ar[r]^(.48){(\thet^h_{M,N})\ps}\ar[d]_{(h\ps)^*(\beta^\psi_{M,N})}&\hm{G}{S}{h^*(M)}{h^*(N)}\ps\ar@{ >->}[d]^{\beta^\psi_{h^*(M),h^*(N)}}\\(h\ps)^*(\hm{H}{R\ps}{M\ps}{N\ps})\ar[r]^(.45){\thet^{h\ps}_{M\ps,N\ps}}&\hm{H}{S\ps}{(h\ps)^*(M\ps)}{(h\ps)^*(N\ps)}}\] in $\grmod^H(S\ps)$. If $\ke(\psi)$ is finite or $M$ is small, then the vertical morphisms are isomorphisms (\ref{a50} A), \ref{c70} A)).\smallskip

C) The morphism of functors $\thet^h$ need be neither a mono- nor an epimorphism. For an (ungraded) counterexample (cf. \cite[II.5 Exercice 1]{a}), let $R=\Z/4\Z$ and $S=\Z/2\Z$, and let $h\colon R\rightarrow S$ denote the canonical projection with kernel $\mathfrak{a}=2\Z/4\Z$. Consider the $R$-modules $M=h_*(S)$ and $N=R$. Then, $h^*(N)\cong S$ (\ref{b20} B)), and since $\mathfrak{a}S=0$ we have $h^*(M)\cong S$. It follows that $\hm{G}{S}{h^*(M)}{h^*(N)}\cong\hm{G}{S}{S}{S}\cong S\neq 0$. There is a non-zero morphism of $R$-modules $u\colon h_*(S)\rightarrow R$ with $u(\overline{1})=\overline{2}$, and we have $\hm{G}{R}{M}{N}=\{0,u\}\neq 0$. Since $\mathfrak{a}u=0$ it follows that $h^*(\hm{G}{R}{M}{N})\cong S\neq 0$. Now, we have $\thet^h_{M,N}(1_S\otimes u)(1_S\otimes\overline{1})=h(u(\overline{1}))=h(\overline{2})=\overline{0}$, hence $\thet^h_{M,N}(u\otimes 1_S)=0$, and thus $\thet^h_{M,N}=0$. In particular, $\thet^h_{M,N}$ is neither a mono- nor an epimorphism.
\end{no}

\begin{prop}\label{d50}
Let $M$ and $N$ be $G$-graded $R$-modules. Then, $\thet^h_{M,N}$ is a mono-, epi- or isomorphism if and only if $\nu^{\Id_R}_{M,N,h_*(S)}$ is so.
\end{prop}

\begin{proof}
By \ref{c25} there is an isomorphism \[h_*(\hm{G}{S}{h^*(M)}{h^*(N)})\overset{\cong}\longrightarrow\hm{G}{R}{M}{h_*(h^*(N))}\] in $\grmod^G(R)$. So, by \ref{b20} B) and the symmetry of tensor products we have a commutative diagram \[\xymatrix@C60pt{h_*(h^*(\hm{G}{R}{M}{N}))\ar[r]^(.47){h_*(\thet^h_{M,N})}\ar[d]_\cong&h_*(\hm{G}{S}{h^*(M)}{h^*(N)})\ar[d]^\cong\\\hm{G}{R}{M}{N}\otimes_Rh_*(S)\ar[r]^{\nu^{\Id_R}_{M,N,h_*(S)}}&\hm{G}{R}{M}{N\otimes_Rh_*(S)}.}\] Since $h_*$ is faithful and conservative (\ref{b10}) this yields the claim.
\end{proof}

\begin{cor}
a) If $h_*(S)$ is projective, then $\thet^h$ is a monomorphism. If $h_*(S)$ is projective and of finite type, then $\thet^h$ is an isomorphism.

b) Let $M$ be a $G$-graded $R$-module. If $M$ is projective and of finite type, or if $M$ is small and $h_*(S)$ is projective, then $\thet^h_{M,N}$ is an isomorphism for every $G$-graded $R$-module $N$.\footnote{For $G=0$, b) generalises \cite[II.5.3 Proposition 7]{a}.}
\end{cor}

\begin{proof}
a) Immediately from \ref{b70} and \ref{d50}.
\end{proof}

\begin{thm}\label{d70}
The following statements are equivalent:\/\footnote{For $G=0$, the implication (i)$\Rightarrow$(v) generalises \cite[II.3.3 Proposition 2 Corollaire]{a} and \cite[II.2.7 Proposition 18]{ac}, while the equivalence (i)$\Leftrightarrow$(ii) is contained in \cite[Th\'eor\`eme 1]{roby}.} (i) $h$ is an epimorphism; (ii) $\sigma^h_S$ is an isomorphism; (iii) $\sigma^h$ is an isomorphism; (iv) $\widetilde{\rho}^h$ is an isomorphism;\linebreak (v) $\gamma^h$ is an isomorphism; (vi) $\eta^h$ is an isomorphism; (vii) $\eta^h_{S,\bullet}$ is an isomorphism.
\end{thm}

\begin{proof}
``(i)$\Rightarrow$(ii)'': Suppose that $h$ is an epimorphism. The morphisms of $G$-graded $R$-algebras $S\rightarrow S\otimes_RS,\;s\mapsto s\otimes 1_S$ and $S\rightarrow S\otimes_RS,\;s\mapsto 1_S\otimes s$ coincide, as their compositions with $h$ coincide. Thus, $\ke(\sigma^h_S)=\langle s\otimes 1_S-1_S\otimes s\mid s\in S^\hme\rangle_S=0$, hence $\sigma^h_S$ is a mono- and therefore an isomorphism (\ref{c40}).

``(ii)$\Rightarrow$(iii)'': For a $G$-graded $S$-module $M$ we have $h_*(\sigma^h_M)\circ\rho^h_{h_*(M)}=\Id_{h_*(M)}$ (\ref{c20}). So, since $h_*$ is conservative (\ref{b10}), $\sigma^h_M$ is an isomorphism if and only if $\rho^h_{h_*(M)}$ is an epimorphism. Suppose now that $\sigma^h_S$ is an isomorphism. Then, $\rho^h_{h_*(S)(g)}$ is an epimorphism for every $g\in G$ (\ref{c10} A)). Let $M$ be a $G$-graded $S$-module. There exist a family $(g_i)_{i\in I}$ in $G$ and an epimorphism $p\colon\bigoplus_{i\in I}S(g_i)\twoheadrightarrow M$ in $\grmod^G(S)$. Keeping in mind \ref{b40} and \ref{c79} a) we get a commutative diagram \[\xymatrix@C80pt{\bigoplus_{i\in I}h_*(S)(g_i)\ar@{->>}[r]^(.42){\bigoplus_{i\in I}\rho^h_{h_*(S)(g_i)}}\ar[d]_\cong&\bigoplus_{i\in I}(h_*(S)\otimes_Rh_*(S)(g_i))\ar[d]^\cong\\h_*(\bigoplus_{i\in I}S(g_i))\ar[r]^(.43){\rho^h_{h_*(\bigoplus_{i\in I}S(g_i))}}\ar@{->>}[d]_{h_*(p)}&h_*(S)\otimes_Rh_*(\bigoplus_{i\in I}S(g_i))\ar@{->>}[d]^{h_*(h^*(h_*(p)))}\\h_*(M)\ar[r]^(.45){\rho^h_{h_*(M)}}&h_*(S)\otimes_Rh_*(M)}\] in $\grmod^G(R)$, where the unmarked morphisms are the canonical ones. It follows that $\rho^h_{h_*(M)}$ is an epimorphism, and thus $\sigma^h_M$ is an isomorphism.

``(iii)$\Rightarrow$(i)'': If $\sigma^h$ is an isomorphism then so is $\sigma^{h_{[0]}}_{S_{[0]}}=(\sigma^h_S)_{[0]}$ (\ref{c10} B)), thus $h_{[0]}$ is an epimorphism by the corresponding ungraded statement (\cite[Th\'eor\`eme 1]{roby}), and therefore $h$ is an epimorphism.

``(iii)$\Rightarrow$(v)'': Let $M$ and $N$ be $G$-graded $S$-modules. Keeping in mind \ref{b20} B) and the associativity of tensor products we get a commutative diagram \[\xymatrix@C80pt{h_*(M)\otimes_Rh_*(N)\ar[r]^{\gamma^h_{M,N}}&h_*(M\otimes_RN)\\h_*(h^*(h_*(M)))\otimes_Rh_*(N)\ar[u]^{h_*(\sigma^h_M)\otimes_Rh_*(N)}&h_*(h^*(h_*(M))\otimes_Sh^*(h_*(N)))\ar[u]_{h_*(\sigma^h_M\otimes_S\sigma^h_N)}\\h_*(S)\otimes_Rh_*(M)\otimes_Rh_*(N)\ar[u]^\cong\ar[r]^\cong&h_*(h^*(h_*(M)\otimes_Rh_*(N)))\ar[u]^\cong_{h_*(\delta^h_{h_*(M),h_*(N)})}}\] in $\grmod^G(R)$, where the unmarked morphisms are the canonical ones. If $\sigma^h$ is an isomorphism, then so are all the vertical morphisms in the above diagram, and thus $\gamma^h_{M,N}$ is an isomorphism, too.

``(v)$\Rightarrow$(iii)'': If $\gamma^h$ is an isomorphism, then so is $\gamma^h_{S,N}=h_*(\sigma^h_N)$ for every $G$-graded $S$-module $N$ (\ref{d20} A)). As $h_*$ is conservative (\ref{b10}) it follows that $\sigma^h$ is an isomorphism.

``(iii)$\Leftrightarrow$(vii)'': For $G$-graded $S$-modules $M$ and $N$ we have a commutative diagram \[\xymatrix{\hm{}{\grmod^G(S)}{M}{N}\ar[r]^(.44){u\mapsto h_*(u)}\ar[d]_{\hm{}{\grmod^G(S)}{\sigma^h_M}{N}}&\hm{}{\grmod^G(R)}{h_*(M)}{h_*(N)}\\\hm{}{\grmod^G(S)}{h^*(h_*(M))}{N}\ar[ru]_\cong&}\] in $\ab$, where the unmarked morphism is given by the adjunction $(h^*,h_*)$ (\ref{c20}). The morphism $\sigma^h$ is an isomorphism if and only if $\hm{}{\grmod^G(S)}{\sigma^h_M}{N}$ is an isomorphism for all $G$-graded $S$-modules $M$ and $N$, hence if and only if the horizontal morphism in the above diagram is an isomorphism for all $G$-graded $S$-modules $M$ and $N$. By \ref{b40} this is equivalent to $\eta^h$ being an isomorphism.

``(iv)$\Leftrightarrow$(vii)'': This follows from the facts that $\eta^h_{S,N}=h_*(\widetilde\rho^h_N)$ for every $G$-graded $S$-module $N$ (\ref{d30} A)) and that $h_*$ is conservative (\ref{b10}). 

``(vi)$\Rightarrow$(vii)'': Clear.

``(vii)$\Rightarrow$(vi)'': Let $M$ be a $G$-graded $S$-module. There exist a free $G$-graded $S$-module $L$ and an epimorphism $p\colon L\twoheadrightarrow M$ in $\grmod^G(S)$, yielding a commutative diagram \[\xymatrix@C40pt{h_*(\hm{G}{S}{M}{\bullet})\ar@{ >->}[r]^(.47){\eta^h_{M,\bullet}}\ar@{ >->}[d]_{h_*(\hm{G}{S}{p}{\bullet})}&\hm{G}{R}{h_*(M)}{h_*(\bullet)}\ar@{ >->}[d]^{\hm{G}{R}{h_*(p)}{h_*(\bullet)}}\\h_*(\hm{G}{S}{L}{\bullet})\ar@{ >->}[r]^(.47){\eta^h_{L,\bullet}}&\hm{G}{R}{h_*(L)}{h_*(\bullet)}.}\] Let $N$ be a $G$-graded $S$-module, and let $u\in\hm{G}{R}{h_*(M)}{h_*(N)}$ be such that the map underlying $u\circ h_*(p)$ is $S$-linear. If $x\in M^\hme$ and $s\in S^\hme$, then there exists $y\in L^\hme$ with $x=p(y)$, and it follows $u(sx)=u(sp(y))=u(p(sy))=su(p(y))=su(x)$. Thus, the map underlying $u$ is $S$-linear. This shows that the above diagram of functors is cartesian. Hence, we may replace $M$ by $L$ and thus suppose that $M$ is free (\cite[08N4]{stacks}). So, there exists a family $(g_i)_{i\in I}$ in $G$ with $M=\bigoplus_{i\in I}S(g_i)$. By \ref{b40} and \ref{c79} a) there is a commutative diagram \[\xymatrix@C70pt@R15pt{h_*(\hm{G}{S}{\bigoplus_{i\in I}S(g_i)}{\bullet})\ar@{ >->}[r]^(.48){\eta^h_{\bigoplus_{i\in I}S(g_i),\bullet}}\ar[d]_\cong&\hm{G}{R}{h_*(\bigoplus_{i\in I}S(g_i))}{h_*(\bullet)}\ar[d]^\cong\\\prod_{i\in I}h_*(\hm{G}{S}{S}{\bullet})(-g_i)\ar@{ >->}[r]^(.47){\prod_{i\in I}\eta^h_{S,\bullet}(-g_i)}&\prod_{i\in I}\hm{G}{R}{h_*(S)}{h_*(\bullet)}(-g_i),}\] where the vertical morphisms are the canonical ones. If $\eta^h_{S,\bullet}$ is an isomorphism, then so is $\prod_{i\in I}\eta^h_{S,\bullet}(-g_i)$, and the above diagram implies that $\eta^h_{M,\bullet}=\eta^h_{\bigoplus_{i\in I}S(g_i),\bullet}$ is an isomorphism, too.
\end{proof}

\begin{cor}\label{d80}
The morphism of $G$-graded rings $h$ is an epimorphism if and only if the morphism of $H$-graded rings $h\ps$ is an epimorphism.
\end{cor}

\begin{proof}
Immediately from \ref{c10} B) and \ref{d70}.
\end{proof}

\begin{cor}
$\sigma^h$, $\gamma^h$, $\eta^h$ and $\widetilde{\rho}^h$ are isomorphisms if and only if $\sigma^{h\ps}$, $\gamma^{h\ps}$, $\eta^{h\ps}$ and $\widetilde{\rho}^{h\ps}$ are so.
\end{cor}

\begin{proof}
Immediately from \ref{d70} and \ref{d80}.
\end{proof}


\noindent\textbf{Acknowledgements.} I am grateful to Jeremy Rickard who helped me to understand the functors $h_+$ and $h^+$ (\url{https://mathoverflow.net/questions/300531}), and to the pseudonymous MO user abx who answered questions about purity and about commutation of $h^*$ with infinite products (\url{https://mathoverflow.net/questions/199721}, \url{https://mathoverflow.net/questions/201022}).


\begin{thebibliography}{99}

\bibitem{bass} H. Bass, \textit{Algebraic K-theory.} W. A. Benjamin, New York, 1968.

\bibitem{a} N. Bourbaki, \textit{\'El\'ements de math\'ematique. Alg\`ebre.} Chapitres 1 \`a 3. Masson, Paris, 1970. Chapitre 10. Masson, Paris, 1980.
 
\bibitem{ac} N. Bourbaki, \textit{\'El\'ements de math\'ematique. Alg\`ebre commutative.} Chapitres 1 \`a 4. Masson, Paris, 1985.

\bibitem{ce} H. Cartan, S. Eilenberg, \textit{Homological algebra.} Princeton Univ. Press, Princeton, 1956.

\bibitem{ct} R. Colpi, J. Trlifaj, \textit{Classes of generalized $*$-modules.} Comm. Algebra 22 (1994), 3985--3995.

\bibitem{gpmn} J. L. G\'omez Pardo, G. Militaru, C. N\u{a}st\u{a}sescu, \textit{When is ${\rm HOM}_R(M,-)$ equal to ${\rm Hom}_R(M,-)$ in the category $R-gr$?} Comm. Algebra 22 (1994), 3171--3181.

\bibitem{ilf} A. Grothendieck, \textit{Introduction au langage fonctoriel.} Lecture notes. Facult\'e de Science d'Alger, 1965.

\bibitem{ks} M. Kashiwara, P. Schapira, \textit{Categories and sheaves.} Grundlehren Math. Wiss. 332. Springer, Berlin, 2006.

\bibitem{lambek} J. Lambek, \textit{Lectures on rings and modules.} Blaisdell, London, 1966.

\bibitem{morita} K. Morita, \textit{Adjoint pairs of functors and Frobenius extensions.} Sci. Rep. Tokyo Kyoiku Daigaku Sect. A 9 (1965), 40--71.

\bibitem{nvo1} C. N\u{a}st\u{a}sescu, F. Van Oystaeyen, \textit{Graded ring theory.} North-Holland Math. Library 28. North-Holland, Amsterdam, 1982.

\bibitem{rentschler} R. Rentschler, \textit{Sur les modules $M$ tels que $\hm{}{M}{-}$ commute avec les sommes directes.} C. R. Acad. Sci. Paris S\'er. A-B 268 (1969), 930--933.

\bibitem{roby} N. Roby, \textit{Diverses caract\'erisations des \'epimorphismes.} S\'eminaire Samuel. Alg\`ebre commutative, tome 2 (1967--1968), exp. n$^{\rm o}$ 3, 1--12.

\bibitem{cihf} F. Rohrer, \textit{Coarsenings, injectives and Hom functors.} Rev. Roumaine Math. Pures Appl. 57 (2012), 275--287.

\bibitem{gic} F. Rohrer, \textit{Graded integral closures.} Beitr. Algebra Geom. 55 (2014), 347--364.

\bibitem{schubert2} H. Schubert, \textit{Kategorien II.} Springer, Berlin 1970.

\bibitem{stacks} The Stacks Project Authors, \textit{Stacks project.} {\tt https://stacks.math.columbia.edu}

\end{thebibliography}
\end{document}